\documentclass{amsart}

\usepackage{graphicx}
\usepackage{amssymb}
\usepackage{amsthm}
\usepackage{amsmath}
\usepackage[overload]{empheq}
\usepackage{euscript}
\usepackage{placeins} 
\usepackage{subfigure}

{\theoremstyle{plain}%
  \newtheorem{theorem}{Theorem}[section]
  
  \newtheorem{proposition}{Proposition}[section]
  \newtheorem{lemma}{Lemma}[section]
  
  \newtheorem{remark}{Remark}
}

\def\R{\mathbb R}

\newcommand{\dd}{{\rm d}}

\usepackage{fancyhdr}
\renewcommand{\headrulewidth}{0pt}

\begin{document}

\fancypagestyle{plain}{ \fancyfoot{} \renewcommand{\footrulewidth}{0pt}}
\fancypagestyle{plain}{ \fancyhead{} \renewcommand{\headrulewidth}{0pt}}

~

  \vskip 2.1 cm   

\centerline {\bf \LARGE  Analytic solutions and numerical method}

\smallskip

\centerline {\bf \LARGE for a  coupled thermo-neutronic problem}

 \bigskip  \bigskip

\centerline { \large    Olivier Lafitte$^{ab}$ and Fran\c{c}ois Dubois$^{acd}$}

\smallskip 

\centerline { \it  \small   
  $^a$   International Research Laboratory 3457 du  Centre National de la Recherche Scientifique,}

\centerline { \it  \small  
Centre de Recherches Math\'ematiques, Universit\'e de Montr\'eal, Montr\'eal, QC, Canada.}

\centerline { \it  \small   
  $^b$ Universit\'e Paris 13, Sorbonne Paris Cit\'e, LAGA, CNRS (UMR 7539),}

\centerline { \it  \small  99 Avenue J.-B. Cl\'ement F-93430, Villetaneuse Cedex, France.}

\centerline { \it  \small   
  $^c$  Conservatoire National des Arts et M\'etiers, LMSSC laboratory, Paris, France. }

\centerline { \it  \small   
  $^d$ Laboratoire de Math\'ematiques d'Orsay, Universit\'e Paris-Saclay, France.}

\bigskip 

\centerline {June 18th, 2028    
  {\footnote {\rm  \small $\,$ This contribution has been presented at LAMA Chambery and CEA Saclay in Spring 2022,
    as well as OPFSA (Montreal) in June 2022. It is published in
    {\it Communications in Mathematical Sciences}, volume 24, issue 6, pages 1745-1707, 2026.}}}

\bigskip 


 
\bigskip 
\noindent {\bf Abstract } 

\noindent
We consider in this contribution a simplified idealized one-dimensional model in a nuclear core reactor coupling
the diffusion equation on the neutron flux with the enthalpy equation for the water which collects the heat produced
by this idealized nuclear core. These equations are coupled through the dependency of the coefficients of the diffusion
equation in terms of the enthalpy. We propose a numerical method treating globally the coupled problem for finding its unique solution. 
Simultaneously, we use incomplete elliptic integrals to represent analytically the density of neutrons and
the enthalpy in the fluid. Both methods lead to the same solution with high accuracy. However, another quantity,
generally used as a benchmark for comparing results, depends considerably on the approximation used
for the coefficients of the diffusion equation.

\bigskip \bigskip \bigskip

\section{Introduction}
\label{sec:intro}
%
\fancyhead[EC]{\sc{Olivier Lafitte and Fran\c{c}ois Dubois}} 
\fancyhead[OC]{\sc{Analytic solutions and numerical method for a  coupled problem}}
\fancyfoot[C]{\oldstylenums{\thepage}}
%
This paper deals with a multiphysics problem which arises in the engineering community of the nuclear fission. We begin by describing the global principles of the two different physics involved in our problem.\\
\underline{The first one} is the neutronics problem for nuclear reactors. Heat release occurs at the occasion of the fission reaction in Uranium (for example). The neutrons in a reactor follow the Boltzmann system of transport equations with collisions. This system is approximated in the neutronicians community by an instationary diffusion equation (see for example Chapter 5 of \cite{lamarsh} and in particular Eq. (5.16-5.17)) on the neutron flux $\phi$ which models the quantity of neutrons that are locally produced in the fission reaction (see Section 3 of \cite{lamarsh} for the study of the microscopic fission and absorption cross sections and Section 6 of the same reference for the instationary diffusion equation Eq. (6.1)). The heat production induced by this reaction is then proportional to $\phi$. 

This instationary diffusion equation on $\phi$ is replaced in the nuclear studies by a stationary diffusion equation (on a similar quantity $\varphi$) where a physically relevant scalar quantity $k$ appears, called the {\bf multiplication factor} (Eq. (6.2) of the same reference). A reactor is said to be {\it critical} if $k=1$ (see Section 6, Eq. (6.2), (6.10), (6.12) of \cite{lamarsh}). The numerical value of $k$ obtained in a calculation or simulation yields a possible criterion for instability of the device which is simulated. 

At each point of the medium, the coefficients of the instationary diffusion equation depend on thermodynamical variables, in particular on the enthalpy of the system. This dependency, in the specialized literature, is called the feedback of the fluid, and an example of this dependency is the Doppler effect (when the temperature of the fluid increases, the cross section 'flattens' see \cite{CEA-monograph}). The coefficients are solely tabulated from the experiments. 

We do not consider, in this elementary model, the evolution equation on the neutron precursors.

\underline{The second one} is the fluid model, from which one may derive the enthalpy hence the temperature.

We use in our study a simple model, called the Low Mach number approximation, which corresponds to a quasi-isobaric situation. This model was proposed by Dellacherie \cite{SD} as a 3d model. The applications considered are axisymmetric problems, and a first model under consideration would be a 1d approximation, where the velocity of the fluid has only one component and the water flow is constant. This yields the simple thermohydraulic model of a transport equation for the enthalpy where the heat source is proportional to the density of neutrons. 

Few theoretical and numerical studies have been done on the coupled physical problem, and all these studies are rather complex. Communities in the industry focused mostly on the coupling of efficient and specialized numerical codes. Among these studies and codes, one may mention Nuresafe (2013-2015) \cite{nuresafe}, or code coupling platforms (Salomé \cite{hemera} and Mouse \cite{mouse}).

In order to construct a global solution of the multiphysics problem, a	simplified model has been proposed by Dellacherie and L. \cite{AML}. It is a model relying on a system of nonlinear differential equations (in dimension 1) with three unknowns: the functions $\varphi$ (density of neutrons) and $h$ (enthalpy of the system) and the multiplication factor $k$ mentioned above. As the equation on the neutron flux is of order 2 and the equation for the enthalpy is of order 1, and as we have four boundary conditions, an equation for the multiplication factor follows. In this paper, we propose to study this equation for the multiplication factor in a particular case for the model of Dellacherie and L.. In explicit cases for the Doppler effect, we calculate analytically $k$ and $\varphi$ and, in the general case, we construct a numerical procedure which gives an extremely precise solution $(k,\varphi)$.

\section{Statement of the problem}
\label{sec:statement}
The instationary diffusion equation on the neutron flux is
\begin{equation}\label{dagger}\frac{1}{v}\partial_t\phi-(D\phi')'+\Sigma_a\phi =\nu \Sigma_f\phi, \mbox{on } (0,L)\times (0, +\infty) \end{equation}

where $v$ is the velocity of the neutrons, $D$ is the diffusion coefficient, $\Sigma_a$ is the absorption cross-section, $\Sigma_f$ is the macroscopic fission cross section, and $\nu$ is the number of neutrons generated by a fission.

It is replaced by:\\
$\bullet$ the stationary diffusion equation on $[0,L]$
\begin{equation}-(D\varphi')'+\Sigma_a\varphi =\frac{\nu \Sigma_f\varphi}{k},\label{neutronics}\end{equation}
where $\varphi$ satisfies the homogeneous Dirichlet condition at $z=0$ and $z=L$, called therein the {\bf neutronics model}, where $k$ is the multiplication factor mentioned in the Introduction\\
$\bullet$ and the simplified equation on the enthalpy, which reads
\begin{equation}\label{thermo}D_eh'=K\varphi,\end{equation}
where $D_e$ is the water flow and $K$ the production of energy for a neutron, and $h$ is given at $z=0$ and $z=L$: $h(0)=h_e, h(L)=h_s$, $h_e$ and $h_s$ given, which is called the {\bf thermohydraulics model}.\\
The coupling of these two equations comes through the dependency of the coefficients of (\ref{neutronics}) in $h$, (through, for example feedback). Note that the feedback is normally translated into a dependency of $\Sigma_a$ in $h$, we consider here a different model for which $\Sigma_a$ constant and $\Sigma_f$ depends on $h$.

Recall first the construction of an analytic solution of the problem (\ref{neutronics})-(\ref{thermo}), subject to the above four boundary conditions,  on $[0, L]$. This problem is  addressed by Dellacherie and L. in \cite{AML}. 

We assume in this study that $L, D, \Sigma_a$, constants, are linked by the relation $D=L^2\Sigma_a$, which allows to simplify this problem by considering $D, \Sigma_a$ and $L$ equal to 1 in the diffusion equation. The constants $D_e$ and $K$ are chosen as 1 in the equation for the enthalpy $D_e h'= K\phi$, $h_e=0$, $h_s=1$ without losing generality (this is indeed a choice of units for the enthalpy). If one needs to keep a general problem (but it is not the purpose of this paper) one can for example write $\epsilon^{-1}= \frac{L^2\Sigma_a}{D}$, dimensionless, and the neutronics equation (\ref{neutronics}) rewrites $-\epsilon \varphi''+\varphi= \frac{\nu \Sigma_f(h)}{\Sigma_ak}\varphi$.

From now on, the domain is $z\in [0, 1]$ and we denote by $\Sigma(h)$ the function $\nu \Sigma_f(h)$. Note that the function $\Sigma$ depends on the variable $h$, and that, in the equation, the coefficient $\Sigma$ depends on the unknown $h: z\longmapsto h(z)$ and appears as $\Sigma(h(z))$ in the ODEs. Call, for simplicity $\lambda:=\frac{1}{k}$.

The ODEs are then, where the unknowns are $\varphi, h, \lambda$:
\vskip 0.3 cm
\noindent $-\varphi''+\varphi= \lambda \Sigma(h)\varphi$ for the stationary diffusion equation (the neutronic model),\\
\noindent$h'=\varphi$ for the enthalpy equation (the thermohydraulic model).
  \vskip 0.3 cm
  The associated boundary conditions are $\varphi(0)=\varphi(1)=0$ (no neutron flux at the boundaries of the domain) and $h(0)=0$, $h(1)=1$ (given difference of enthalpy between the extremities of the domain). As the system is of order 3 with four boundary conditions, this system yields a condition on $\lambda$.

The final system is
\begin{equation} \label{modele-couple} \left \{ \begin {array}{l}
    -\varphi''(z)+ \varphi(z) = \lambda \, \Sigma \big(h(z) \big) \,  \varphi(z) ,\,   h'(z) =  \varphi(z)   \,, \,
    0 < z < 1 \,, \\
    h(0) = 0 , \, h(1) = 1 , \, \varphi(0) = 0 ,\, \varphi(1) = 0 ,\, \lambda \geq 0 ,\, \varphi(z) > 0 \,\, {\rm if} \,\, 0 < z < 1 .   
  \end{array} \right. \end{equation}

We assume throughout this paper
  \begin{equation} \label{hypothese-Sigma} 
    \Sigma \,\,\,  {\rm continuous}, \,\,     \Sigma(h)\geq \Sigma_*>0, \,\,  {\rm for} \,\, 0\leq h\leq 1.
  \end{equation}

Let $\psi_{\lambda}$ solution of 

\begin{equation}\label{psilambda}\psi_{\lambda}''(X)=2-2\lambda \Sigma(X),\psi_{\lambda}(0)=\psi_{\lambda}(1)=0.\end{equation} 

Simple algebra, using the boundary conditions $\varphi(0)=0, \varphi(1)=0, h(0)=0, h(1)=1$ yields the equation $(\varphi(z))^2=\psi_{\lambda}(h(z))$, that is
$$(h'(z))^2= \psi_{\lambda}(h(z)).$$
For reader's convenience, we reproduce it here. One has
\begin{equation}\begin{array}{ll}-2\varphi''(z) + \psi_{\lambda}''(h(z))\varphi(z)=0&\rightarrow -2h'' +\psi_{\lambda}'(h(z))=c_0\\
&\rightarrow -2h''h'+h'\psi_{\lambda}'(h)= (c_0h+c_1)'\\
&\rightarrow -(h')^2+\psi_{\lambda}(h(z))=0,\end{array}\label{aaa}\end{equation}
the last equation being obtained by using the boundary conditions on $h, \varphi$.
As $\varphi\geq 0$, $h$ is increasing hence $\frac{h'(z)}{\sqrt{\psi_{\lambda}(h(z))}}=1$. This yields, for all $z\geq 0$ such that $0\leq h(z)\leq 1$,
$$\int_0^{h(z)}\frac{dh}{\sqrt{\psi_{\lambda}(h)}}=z, \varphi(z)= \sqrt{\psi_{\lambda}(h(z))}.$$
The equation on $\lambda$ is then
\begin{equation}I_{\lambda}:=\int_0^1\frac{dh}{\sqrt{\psi_{\lambda}(h)}}=1.\label{equation-k}\end{equation}
It has been proven in \cite{AML} that this equation has a unique solution $\lambda_*>0$ which implies
   \begin{lemma}
  \label{lem:hstar}
  System (\ref{modele-couple}) has a unique solution $(\lambda_*, h_*, \phi_*)$ where $h_*\in C^1([0,1])$, $\phi_*\in C^2([0,1])$.
  \end{lemma}
  
  The aim of this paper is to present, in the case where one knows three values of the function $\Sigma$, an analytic and symbolic method which finds {\bf exact} solutions of (\ref{modele-couple}) using the incomplete Jacobi functions (which are its solutions) when $\Sigma$ is a polynomial of degree less than 2. This is not unrealistic; in the community of researchers in the nuclear industry, few different values of the cross sections are known in a simulation for a given case. These analytic examples are important as well: they allow uncertainty studies on these values of the cross sections. 
  
  We also present a numerical method which solves (\ref{modele-couple}) without using the numerical methods traditionally used for solving each ODE but rather concentrating on solving the equation (\ref{equation-k}). 

The numerical and analytic methods both allow to show that the result (value of $\lambda_*$ and graph of $\varphi$) does not depend much on the  interpolation chosen for the function $\Sigma$, however the precision demanded by benchmarks is smaller than these variations based on the interpolation method (see conclusion).
  
  This numerical method, as well as analytic and symbolic methods, are implemented when one knows only three values of $\Sigma$, and for simplicity again one assumes that one knows $\Sigma(0)=\sigma_0$, $\Sigma(\frac12)=\sigma_{\frac12}$, $\Sigma(1)=\sigma_1$.

\section{Analytical approach}
\label{sec:analytical} 

We study in this section four representations of the function $\, \Sigma \, $ which lead
to exact analytical solutions of (\ref{modele-couple}) using the incomplete elliptic integrals.
It corresponds to cases for the values $\sigma_0, \sigma_{1/2} , \sigma_1$  and four cases for the function $\Sigma$: \\
$\bullet$ $\, \Sigma \, $ constant (called Zero-th case later) ($  \sigma_0 = \sigma_{1/2} = \sigma_1 $),\\
$\bullet$ $\, \Sigma \, $ an affine polynomial ($ \sigma_{1/2} = {1\over2} \, (\sigma_0 + \sigma_1) $) (First case), \\
and, when $ \sigma_{1/2} \not= {1\over2} \, (\sigma_0 + \sigma_1) $, we consider two cases:\\
$\bullet$ $\, \Sigma \, $ the interpolation polynomial of degree 2 defined by the three input data (Second case),\\
$\bullet$ $ \, \sigma_0 $, $ \, \sigma_{1/2} \, $ and $ \, \sigma_1 \, $ and $ \,  \Sigma \, $
continuous piecewise affine defined by the three previous input data  (called Third case in the sequel). \\
The model considered here is a toy example where $\Sigma$ is known at the three points $h=0, h=1/2$ and $h=1$ and has the respective values $\sigma_0, \sigma_{1/2} , \sigma_1$. Our analytic approach is easily generalized to any situation where $\Sigma(h_1), \Sigma(h_2), \Sigma(h_3)$ are given for $0\leq h_1<h_2<h_3\leq 1$. It can be extended {\it e.g.} Abramowitz and Stegun \cite{AS} to the case $\Sigma(h_p)$ given for $0\leq h_1<\cdots<h_N\leq 1$, $1\leq p\leq N$ but this is not the purpose of our study. 

Following  \cite{AS}, the incomplete elliptic integral of the first kind
$\, K(m,\, \varphi) \, $ is 
\begin{equation*} K(m,\, \varphi) = \int_0^{\,\varphi} {{\dd \theta}\over{\sqrt{1 - m \, \sin^2 \theta}}} \,, \,\,
  0 \leq \varphi \leq {{\pi}\over{2}} \,,\,\,  m < 1 .   \end{equation*}
The complete elliptic integral of the first kind $ \, K(m) \, $ is defined by

\centerline{ $ 
  K(m) \equiv K(m,\, {{\pi}\over{2}} ) \, . $}

In (\ref{modele-couple}),
we assume that $ \, \Sigma \, $ is known only through the
three positive real numbers $ \, \sigma_0 $, $ \, \sigma_{1/2} \, $ and $ \, \sigma_1 \, $
which are respectively the values of $ \, \Sigma \, $ at $ \, 0 $, $\, 1/2 \, $ and $\, 1 $. 
  In what follows, we establish that for four modelling of the function $ \, \Sigma \, $ from these values,
  it is possible to put in evidence an analytical approach to determine firstly the scalar  parameter $ \, \lambda > 0 \, $ and secondly
  the functions $ \,  z \longmapsto \varphi(z) \, $ and $ \, z \longmapsto h(z) $. 
  
  Although an efficient numerical method exists for solving $I_{\lambda}=1$ (Section \ref{sec:numerical} describes this method, see also Dubois \cite{Du24}), it is important as well to have analytical formulae for the solution of $I_{\lambda}=1$, denoted in each case by $\lambda(\sigma_0, \sigma_{1/2}, \sigma_1)$. The four expressions of $\lambda(\sigma_0, \sigma_{1/2}, \sigma_1)$, to which we add the expression obtained from the semi-anaytical method of Dellacherie et al. \cite{SYNASC16} are analytic through elliptic integrals. Precise studies of $\lambda(\sigma_0, \sigma_{1/2}, \sigma_1)$ are possible, in particular for uncertainty quantification \cite{LMoussa}: derivatives of elliptic functions are known.
  
 Note that these Jacobi incomplete elliptic functions also appear in another problem of interest, namely the ion distribution in a solute separated by two electrodes, which is a model problem widely studied for the electrokinetics of biological cells. Preliminary results for both problems were presented as a single contribution  at the 16th International Symposium on Orthogonal Polynomials, Special Functions and Applications (June 2022, Montreal, online).

The aim of this Section is to prove the following Theorem:
\begin{theorem}
\label{theorem-analytical}
Let $\lambda, \sigma_0, \sigma_{\frac12}, \sigma_1$ strictly positive real numbers given. For $\Sigma(X)$ the interpolation polynomial satisfying $\Sigma(j)=\sigma_j$ (cases denoted by zero-th case, case 1 and case 2), there exists two constants $C(\lambda, \sigma_0, \sigma_{\frac12}, \sigma_1)>0, m(\lambda, \sigma_0, \sigma_{\frac12}, \sigma_1)\in [0, 1)$ such that
$$I_{\lambda}= C(\lambda, \sigma_0, \sigma_{\frac12}, \sigma_1)K(m(\lambda, \sigma_0, \sigma_{\frac12}, \sigma_1)).$$
If $\Sigma$ is the piecewise affine function interpolating $\sigma_0, \sigma_{\frac12}, \sigma_1$, there exists four constants $C_-,C_+, m_-, m_+$ such that
$$I_{\lambda}= C_-(\lambda, \sigma_0, \sigma_{\frac12}, \sigma_1)K(m_-(\lambda, \sigma_0, \sigma_{\frac12}, \sigma_1))+C_+(\lambda, \sigma_0, \sigma_{\frac12}, \sigma_1)K(m_+(\lambda, \sigma_0, \sigma_{\frac12}, \sigma_1)).$$
The analytic formulae are given in the subsequent propositions.
\end{theorem}
\begin{proof}\noindent \\
$\bullet$ Zero-th case:  $ \, \Sigma \, $ is constant.

Then the system modelled by the previous set of equations 
is totally decoupled and an exact solution can be provided with elementary arguments. 
\begin{proposition} \label{P2-cas-zero-Sigma-constant}
  If $ \, \sigma_0 = \sigma_{1/2} =  \sigma_1 \equiv \mu > 0 $, the  problem (\ref{modele-couple}) admits a unique decoupled solution:
  we have

  \centerline {$  \lambda \, \mu = 1 + \pi^2 $, $ \, \varphi(z) = {{\pi}\over{2}} \, \sin(\pi \, z) \, $ and
  $\, h(z) = {1\over2}\, (1 - \cos(\pi \, z) ) $.} 
\end{proposition} 
\begin{proof} 
  From the relation $ \,    -\varphi'' = (\lambda \, \mu-1)\,  \varphi(z) \, $ and the conditions
  $ \, \varphi(0) =$ $ \varphi(1) = 0 $ with the constraint  $ \, \varphi(z) > 0 \, $ if $ \, 0 < z < 1 $,
  we deduce that $ \, (\lambda \, \mu-1) \, $ is the first eigenvalue of the Laplace equation
  on the interval $ \, (0,\, 1) \, $ wih Dirichlet boundary conditions. Then $\, \lambda \, \mu-1 =  \pi^2 \, $ and
  $ \,  \varphi(z) = C  \, \sin(\pi \, z) \, $ for some constant $\, C $. We integrate with respect to $ \, z \, $ this relation and we get, using $h(0)=0$,
  $ \, h(z) = {{C}\over{\pi}} \,(1- \cos(\pi \, z) )\,$ for $ \, 0 \leq z \leq 1 $. The  datum $\,  h(1) = 1 \, $ implies $ \, {{2\, C}\over{\pi}} = 1 $. \end{proof}

If the function $\, \Sigma \, $ is no more constant, it has been proven in \cite{SYNASC16,DL}
that the unknown of the problem can be obtained with the following process. First integrate twice the function $ \, \Sigma \, $ and
obtain a convex negative function $ \, V \, $ such that
\begin{equation}\label{Vfunction}
    {{\dd^2 V}\over{\dd h^2}} = \Sigma(h), V(0) = V(1) = 0.
    \end{equation}

  \noindent 
Second, recall that $\psi_{\lambda}$, defined by (\ref{psilambda}), is given by $ \, \psi_\lambda(h) \equiv h \, (h-1) - 2 \, \lambda \, V(h) $. Then the equation for the function $ \, z \longmapsto h(z) \, $
is $ \, h'(z) = \sqrt{\psi_\lambda(h(z))} \, $ and the conditions $ \, h(0)=0, h(1) = 1 \, $ give a scalar equation
for the unknown $ \, \lambda > 0 $:
\begin{equation}\mbox{ Find } \lambda \mbox{ such that } \,  \int_0^{\,1}\!\!  {{\dd h}\over{\sqrt{\psi_\lambda (h)}}} = 1.\end{equation}

A first step is to compute the integral
$ \, I_\lambda \equiv  \int_0^{\,1}\!\!  {{\dd h}\over{\sqrt{\psi_\lambda (h)}}} $. 

Once it is done, we can solve easily the equation $ \, I_\lambda = 1 \, $
with a Newton-like algorithm. When $ \, \lambda \, $ is determined, the explicitation of the functions 
$ \,  z \longmapsto \varphi(z) \, $ and $ \, z \longmapsto h(z) \, $ is not difficult.
Thus the method we propose is founded on an analytical determination of the integral $ \, I_\lambda $.
We focus our attention to this question in the next sub-sections. A by-product of these calculations are the expressions of $h$ and $\varphi$.

The proof of Theorem \ref{theorem-analytical} relies on Propositions \ref{P2-cas-Sigma-affine} and \ref{homographic-change} that follow.

$\bullet$ The first case where there is coupling is the case where $ \, \Sigma \in P_1 $.

In this case, the function $ \, \Sigma \, $ is a positive affine function on the interval $ \, (0,\, 1 ) $.
We set $ \, \sigma_0 = \mu\, (1-\alpha) $, $\, \sigma_{1/2} = \mu \, $ and $ \, \sigma_1 = \mu\, (1+\alpha) $.
Then $ \, \mu > 0 \, $ and $ \, | \alpha | < 1 \, $ to satisfy the constraint of positivity.
We introduce the notation \begin{equation}\label{xi} \, \xi \equiv  \lambda\,\mu.\end{equation}

\begin{proposition} \label{P2-cas-Sigma-affine}
  In the case $ \, \Sigma \in P_1  $, the functions $ \, V \, $ and $ \, \psi_\lambda \, $
  admit the algebraic expressions 

  \centerline {$
 V(h) = \mu \, h \, (h-1) \, \big( {1\over2} - {{\alpha}\over{6}} + {{\alpha}\over{3}} \, h \big) \,,\,\,   
  \psi_\lambda(h) =  h \, (h-1) \, \big( 1-\xi \, (1-{{\alpha}\over{3}}) - {{2\, \alpha\, \xi}\over{3}} \, h \big)  $.}

  \noindent
  Then we have $ \, I_\lambda = {{2}\over{\sqrt{\xi \, \big( 1+ {{|\alpha|}\over{3}} \big) - 1}}} \, K(m) \, $ with
  $ \, m = {{2 \, |\alpha| \, \xi}\over{3 \, \xi + |\alpha| \, \xi - 3}} $. 
\end{proposition} 
\begin{proof}
It is easy to check from these expressions that $V''(h)=\mu[2\alpha h+1-\alpha]=\Sigma(h)$ and that $\psi_\lambda$ satisfies (\ref{psilambda}). The change of variable $h=\sin^2\theta$ for $\alpha<0$ or $h=\cos^2\theta$ for $\alpha>0$ in the integral $ \, I_\lambda$ yields the expression. \end{proof}

\bigskip \noindent $\bullet$ 
Second case:  $ \, \Sigma \in P_2 $.  

In this case, the polynomial $ \, \psi_\lambda \, $ is a polynomial of degree $4$  with real coefficients
and  $ \, \psi_\lambda \, $ is positive on the interval $ \, (0,\, 1 ) $.
We have also $ \, \psi_\lambda (0) = \psi_\lambda (1) = 0 $. 
Recall that  $ \, \psi_{\lambda}(h)=h\,(h-1) - 2\, \lambda \, V(h) \, $ with $ \, V'' \equiv \Sigma \, $ the Lagrange interpolate polynomial
such that $ \, \Sigma(0) = \sigma_0 $,  $ \, \Sigma({1\over2}) = \sigma_{1/2} \, $ and  $ \, \Sigma(1) = \sigma_1 $. All these coefficients are supposed
positive: $ \, \sigma_0 > 0 $, $ \, \sigma_{1/2}  > 0 $, $ \, \sigma_{1} > 0 $.
We introduce appropriate parameters $\, \mu $, $ \, \alpha \, $ and $ \, \delta \, $ such that  $\,  \sigma_0 = \mu \, (1 - \alpha) $,
$ \, \sigma_{1/2} =  \mu \, (1 + \delta)  \, $ and $\,  \sigma_1 = \mu \, (1 + \alpha) $.
In this sub-section,  we exclude the   case of a linear interpolation, {\it id est} $\, \delta \not= 0 $,
  in coherence with the hypothesis that the degree of the polynomial $ \, \psi_\lambda \, $
is exactly equal to 4. 

The first part of this case is the transformation of this polynomial of degree 4:

\begin{lemma} \label{P2-lemme-1}
  With the above  notations and properties, the parameters $\, \mu $, $ \, \alpha \, $ and $ \, \delta \, $ are given by
  $ \, \mu = {1\over2} \, (\sigma_0 +  \sigma_1) $, $\, \alpha = 1 - {{\sigma_0}\over{\mu}} \, $ and 
   $ \, \delta = - {{ \sigma_0 - 2 \, \sigma_{1/2} +  \sigma_1}\over{\sigma_0 +  \sigma_1}}  $.
  They satisfy the inequalities $\, \mu > 0 $, $\, | \alpha | < 1 \, $ and  $ \, \delta > -1 $. If we set
  $ \, \gamma \equiv 1 - {1\over3} \,  | \alpha | + {2\over3} \, \delta $, we have the inequality $ \, \gamma > 0 $. 
\end{lemma}
\begin{proof}
  The two first inequalities for $ \, \mu\, $ and $ \, \alpha \, $ are clear because $ \,  \sigma_0 > 0 \, $ and $ \, \sigma_{1} > 0 $. From
  $ \,  \sigma_{1/2}  > 0 \, $ and $ \,  \mu > 0 \, $ we deduce that $ \, \delta > -1 $.
  Then $ \, \gamma > 1 -  {1\over3} - {2\over3} = 0 $. 
\end{proof}
\begin{lemma} \label{P2-lemme-2}
  With the above  notations and properties and $\xi$ given by (\ref{xi}); then $\, \xi > 0 $. Introduce the two roots $ \, p \, $ and $ \, g \, $ distinct from $0,1$  of the polynomial  $\, \psi_\lambda  $. There exists a constant $a_0$ such that $ \,  \psi_\lambda (h) = a_0 \, h \, (h-1) \, (h-p) \, (h-g) $. One has  $\, a_0 = {2\over3} \, \xi \, \delta $, $ \, a_0 \,(p+g) =  {2\over3} \, \xi \, (\delta + \alpha) \, $
  and $ \,a_0 \, pg = 1 - \xi +  {1\over3}\, \alpha \, \xi -  {2\over3} \, \delta \, \xi $.   
\end{lemma}
\begin{proof}
 The equality $\psi''_{\lambda}(h)=2(1-\lambda \Sigma(h))$ yields, along with $\Sigma(h)= 2\mu(1-\alpha)(h-\frac12)(h-1)+4\mu(1+\delta)h(1-h)+2\mu(1+\alpha)h(h-\frac12)$, the values of $a_0, p+g, pg$ of Lemma \ref{P2-lemme-2}, using elementary algebra. We have used the Sagemath \cite{sage} package, an open source mathematical software system.  In particular, the term of degree 2 of polynomial $\Sigma(h)$ is $2\mu (1-\alpha)-4\mu(1+\delta)+2\mu(1+\alpha)= -4\mu \delta$ which yields the value of $a_0$.
  The inequality $ \, \xi > 0 \, $ is a consequence of the hypothesis $ \, \lambda > 0 \, $ and of the property $ \, \mu > 0 \, $
established in Lemma~\ref{P2-lemme-1}. An abridged version of the following analytical results have been presented in the SYNASC21 conference \cite{DL21}.
\end{proof}

\begin{lemma} \label{P2-lemme-3}
  With the above  notations and properties, if the integral $ \, I_{\lambda} \equiv \int_0^{\,1}\!\!  {{\dd h}\over{\sqrt{\psi_\lambda (h)}}} \, $
  is finite, the roots $ \, p \, $ and $ \, g \, $ cannot be equal to 0 or 1. 
\end{lemma}
\begin{proof}
  If either  $ \, p \, $ or  $ \, g \, $ is equal to 0 or 1, the polynomial $ \, \psi_\lambda (h) \, $ has a double root and the function
  $ \, (0,\,1) \ni h \longmapsto  {{1}\over{\sqrt{\psi_\lambda (h)}}} \, $ is not integrable on the interval $ \, (0,\, 1 ) $.   
\end{proof}

For further purposes, set
\begin{equation}\label{discriminant}
\Delta \equiv (p+g)^2 - 4 \,pg= {{1}\over{a_0^2}} \, \xi \, \big[ (5 \, \delta^2 + \alpha^2 + 6 \, \delta)\, \xi - 6 \, \delta \big] \,, 
\end {equation}

As the polynomial $\psi_{\lambda}$ is exactly of degree 4 and two real roots, the two other roots are either real or complex conjugate, denoted by $p$ and $g$. The main argument in the sequel is expressed in the following transform of the function $\psi_{\lambda}$:
 \begin{proposition}
\label{homographic-change}
If $p+g=1$, then $\psi_{\lambda}(h)= a_0 (v^2-\frac14)(v^2-\frac{\Delta}{4})$ with $v=h-\frac12$ and $-\frac12\leq v\leq \frac12$.

If $p+g\not=1$, there exists a homographic change of variable $T(h)= \frac{h-d}{h-c}$, and $a, b$ such that the roots $p, g, 0, 1$ are transformed into $\pm a, \pm b$ or $\pm a, \pm i b$, with $0<a<1<b$ and
\begin{enumerate}
\item when $p<0<1<g$, $\psi_{\lambda}(h)=a_0(h-c)^4\frac{(T^2-a^2)(b^2-T^2)}{(1-a^2)(b^2-1)}$, $-b\leq T\leq -a$ for $0\leq h\leq 1$,

\item when $p, g$ are of the same sign, $\psi_{\lambda}(h)= \vert a_0 \vert  (h-c)^4\frac{(a^2-T^2)(b^2-T^2)}{(1-a^2)(b^2-1)}$, with $-a\leq T\leq a$ for $0\leq h\leq 1$,
\item when $p$ and $g$ are complex conjugates, $\psi_{\lambda}(h)= \vert a_0\vert (h-c)^4\frac{(a^2-T^2)(b^2+T^2)}{(1-a^2)(b^2+1)}$,  in which case $T(p)= ib$, $T(g)=-ib$, $-a\leq T\leq a$ for $0\leq h\leq 1$.
\end{enumerate}
The expressions of $a, b, c, d$ are obtained in Lemmas \ref{P2-proposition-1}, \ref{P2-proposition-4}, \ref{P2-proposition-3} that follow.

\end{proposition}
\begin{proof} of Proposition \ref{homographic-change}.
The proof of the first line comes from $h(h-1)= v^2-\frac14$ and $h^2-(p+g)h+pg= h^2-h+pq=v^2-\frac{\Delta}{4}$.  

If $p+g\not=1$ the homographic transform needs to be derived. The proof of Proposition \ref{homographic-change} follows from the following Lemmas and their proofs, where we solve for $a, b, c, d$, $c$ and $d$ being the parameters of the homographic function and $0<a<1<b$ are given through the homographic relations (\ref{hom-1}), (\ref{hom-2}), (\ref{hom-3}), (\ref{hom-4}). 


When $\Sigma$ is in $P_2$, the results of Proposition \ref{homographic-change} imply, using Lemma \ref{integrales-elliptiques}, the result of the first item of Theorem \ref{theorem-analytical}, using extensively the identity, for all $(h, {\tilde h})$ different from $c$:
\begin{equation}\label{homographic}\frac{h-{\tilde h}}{h-c}=\frac{T({\tilde h})-T(h)}{T({\tilde h})-1}.\end{equation}\\
For $p+g\not=1$, the relation $\frac{dT}{dh}= \frac{d-c}{(h-c)^2}$ yields
\begin{enumerate}
\item for $p<0<1<g$, $$\begin{array}{ll}\int_0^1\frac{dh}{\sqrt{(h-c)^4((T(h))^2-a^2)(b^2-(T(h))^2)}}&=\frac{2}{d-c}\int_{-b}^{-a}\frac{dT}{\sqrt{(T^2-a^2)(b^2-T^2)}}\\
&=\frac{1}{(d-c)b}K(1-\frac{a^2}{b^2}),\end{array}$$
\item for $p<g<0<1$ or $0<1<p<g$, $$\begin{array}{ll}\int_0^1\frac{dh}{\sqrt{(h-c)^4(a^2-(T(h))^2)(b^2-(T(h))^2)}}&=\frac{1}{d-c}\int_{-a}^{-a}\frac{dT}{\sqrt{(a^2-T^2)(b^2-T^2)}}\\&=\frac{1}{2(d-c)b}K(1-\frac{a^2}{b^2}),\end{array}$$
\item for $p$ and $g$ complex conjugates, $$\begin{array}{ll}\int_0^1\frac{dh}{\sqrt{(h-c)^4(a^2-(T(h))^2)(b^2+(T(h))^2)}}&=\frac{1}{d-c}\int_{-a}^{-a}\frac{dT}{\sqrt{(a^2-T^2)(b^2+T^2)}}\\
&=\frac{2}{(d-c)\vert b\vert }K(-\frac{a^2}{b^2})= {{2 K(\frac{a^2}{b^2}\frac{a^2}{a^2+b^2})}\over{|b|(d-c)\sqrt{b^2+a^2}}} \,,\end{array}$$
\end{enumerate}
and for $p+g=1$, $$\int_{-\frac12}^{\frac12}\frac{dv}{\sqrt{(v^2-\frac14)(v^2-\frac{\Delta}{4})}}=\frac{4}{\sqrt{\vert \Delta\vert}}K(\frac{1}{\Delta}),$$ using either (\ref{cas-racines-plus-grandes-que-un-ou-negatives}) of Lemma \ref{integrales-elliptiques} when $\Delta>0$ or (\ref{cas-sigma-egal-un}) of Lemma \ref{integrales-elliptiques} when $\Delta<0$.

This yields, using the expression of $d-c$ in each case:\\
\noindent

 $I_{\lambda}=2 \,  \frac{(b-a)\sqrt{(1-a^2)\, (1+b^2)}}{\sqrt{|a_0|} \, (a+1)\, (b+1)} \, K(1-{{a^2}\over{b^2}}) $ when $p<0<1<g$ (see (\ref{formula abAB}) for $a,b$),
\noindent

 $I_{\lambda}=2 \,  {{2a\sqrt{(1-a^2)\, (1+b^2)}}\over{\sqrt{|a_0|}\, |b| \, (1-a^2)}} \, K({{a^2}\over{b^2}}) $ when $p<g<0<1$ or $0<1<p<g$, $b$ solution of (\ref{eq:b}) and $ \, a = b\, (p-g)+p+g-1 \, $,
\noindent

$I_{\lambda}=2 \,  {{(2c-1)\sqrt{(1-a^2)\, (1+b^2)}}\over{\sqrt{|a_0|}\, |b \, (1+c)|}} \, K(- {{a^2}\over{b^2}} ) $ when $p$ and $g$ are complex conjugate, $c$ is solution of (\ref{eq:c}) and $a=\frac{1}{1-2c}, b=\frac{\sqrt{-\Delta}}{2c-p-g}$,
\noindent

$I_{\lambda}= \frac{4}{\vert a_0\vert \sqrt{\vert \Delta\vert}}K(\frac{1}{\Delta})$ when $p+g=1$.

 \begin{lemma} \label{P2-proposition-1}
  We keep active the above notations. 

  (i) If the discrete positive family $ \, (\sigma_0 ,\,  \sigma_{1/2} ,\, \sigma_1) \, $ is the trace of a concave function,
  {\it id est} if $ \,\,  \sigma_{1/2} - \sigma_0 > \sigma_1 - \sigma_{1/2} \,$, we have $ \, \delta > 0 $, $ \, a_0 > 0 \, $ and $ \, \gamma \, \xi > 1 $.
  The two roots of the function $ \, \psi_\lambda \, $ different from $0,1$
  are real  with opposite signs, that is $ \, p < 0 < 1 < g $ with $ \, p = {1\over2} \, (p+g - \sqrt{\Delta}) \, $ and  $ \, g = {1\over2} \, (p+g + \sqrt{\Delta}) \, $.

  (ii) Conversely, if  the function $ \, \psi_\lambda \, $   has real zeros  with opposite signs, then

\centerline {$
  \sigma_{1/2} - \sigma_0 > \sigma_1 - \sigma_{1/2} $.}
  
 In this case (which is the first item of Proposition \ref{homographic-change}), we deduce the expression 
 
 \centerline{ $\psi_{\lambda}(h)= \vert a_0\vert (h-c)^4\frac{(T^2-a^2)(b^2-T^2)}{(1-a^2)(b^2-1)}$, positive for $-b\leq T\leq -a$}
 
  where
   \begin{equation}\label{formula abAB} \, a =  {{\sqrt{{g\, (g-p)}\over1-p}-1}\over{\sqrt{{g\, (g-p)}\over1-p}+1}} \, ,\,  b =  {{\sqrt{{(1-p)\, (g-p)}\over{g}}+1}\over{\sqrt{{(1-p)\, (g-p)}\over{g}}-1}}. \end{equation}

\end{lemma}
\begin{proof}
  (i) Recall that $ \, a_0 > 0 \, $ because $ \, \delta > 0 \, $ in this case.\\
  Set $ \, f(h) \equiv (h - p) \, (h-g) $. Then
  $ \, \psi_\lambda (h) = -a_0 \, h \, (1-h) \, f(h) \, $ and the function $ \, f \, $ must be negative on $ \, (0,\, 1 ) \, $ because $ \, \psi_\lambda \, $
  is positive on this interval.\\ But
  $     a_0 \, f(0) = a_0 \, pg = 1 - \xi \,  \big( 1 -  {1\over3}\, \alpha +  {2\over3} \, \delta \big)  \,,\,\, 
     a_0 \, f(1) =  1 - \xi \,  \big( 1 +  {1\over3}\, \alpha +  {2\over3} \, \delta \big)  $. 
Then 
  $ \, \xi \, \inf \big( 1 -  {1\over3}\, \alpha +  {2\over3} \, \delta ,\, 1 +  {1\over3}\, \alpha +  {2\over3} \, \delta \big) > 1 \, $ and 
  in other words  $ \, \gamma \, \xi > 1 $. The polynomial $ \, f \, $ is strictly negative on the interval  $ \, (0,\, 1 ) \, $
  and the inequality $ \,  p < 0 < 1 < g \, $ is just a notation that distinguish  $ \, p \, $ as the ``petite'' root and
  $ \, g \, $ as the ``grande'' root.

  (ii) Conversely, if the polynomial $ \, f \, $ is strictly negative on the interval  $ \, (0,\, 1 ) $, we must have $ \, a_0 > 0 \, $ because  the function
    $ \, \psi_\lambda (h) \, $ is positive on this interval. Then $ \, \delta > 0 \, $ and  $ \,\,  \sigma_{1/2} - \sigma_0 > \sigma_1 - \sigma_{1/2} \,$. 

Obtaining the homographic transform follows from the book \cite{AS}, in paragraph 17.8, Example 1 (where they obtain the homographic mapping when writing a polynomial of degree 4 with 4 real roots as a product). Our presentation here is original, because one enforces the homographic mapping and identify the focal points instead of deriving it. 
Lemma \ref{P2-proposition-1} is the case $p<0<1<g$ and the homographic transform satisfies
  \begin{equation}\label{hom-1}\, T(p) = b ,\, T(0) = -b , \, T(1) = -a \, ,  \, T(g) = a ,\end{equation}
   as presented on Figure \ref{fig-p-negatif-g-positif}, and (\ref{homographic}) applied in $\psi_{\lambda}(h)= a_0(h-0)(h-1)(h-p)(h-g)$ yields 
   

\centerline {$ 
  \psi_\lambda (h) = a_0 \, (h-c)^4 \, {{((T(h))^2-a^2)\, (b^2-(T(h))^2)}\over{(1-a^2)\, (b^2-1)}} \, . $}

\noindent
If we enforce the condition $ \, 0 < a < 1 < b $,
we still have $ \, \psi_\lambda (h) > 0 \, $ for $\, 0 < h < 1 \, $ and the transformed integral has to be computed
on the interval $ \, (-b ,\, -a ) $. 

We have now to determine the four parameters $ \, a $, $ \, b $, $ \, c \, $ and $ \, d \, $ as a function of the data $ \, p \, $ and $ \, g $.
We have the constitutive relations, from $c-d=(h-c)(T(h)-1)$, that we apply to $h=p, 0, 1, g$:

\centerline {$
  c - d = (b-1) \, (p-c) = (-b-1) \, (0-c) = (-a-1)(1-c) = (a-1)\, (g - c) \, .$}

\noindent 
From the third equality, we obtain $ \, c = -{{a+1}\over{b-a}} < 0 $. We report this value inside the second and fourth equalities to obtain a system
of two equations for the parameters $ \,a \, $ and $ \, b $: $$\left\{\begin{array}{l}\, p \, (b-a) + a + 1 = - {{(a+1)\, (b+1)}\over{b-1}} \, \\
 \,  g \, (b-a) + a + 1 = - {{(a+1)\, (b+1)}\over{a-1}}.\end{array}\right.$$ Taking the difference, we have
$ \, - {{a+1}\over{a-1}} \, {{b+1}\over{b-1}} = g-p > 0 $. Then $ \, A \equiv  {{a+1}\over{a-1}} \, $ is negative, 
$\,  B \equiv  {{b+1}\over{b-1}} \, $ is positive and $ \, - A \, B = g-p $. 
Using $a+1=\frac{2A}{A-1}$ and $b+1=\frac{2B}{B-1}$  which yields $\frac{(a+1)(b+1)}{a-1}= \frac{2AB}{B-1}= -2\frac{g-p}{B-1}$, and $b-a=\frac{2(A-B)}{(A-1)(B-1)}$, we insert the transformed parameters $ \, A \, $ and $ \, B \, $ into the equation $ \,  g \, (b-a) + a + 1 = - {{(a+1)\, (b+1)}\over{a-1}} \, $ to obtain $\, -A \, (1-p) = g \, B $. Since $ \, A<0 \, $ and $ \, B > 0 $, then  
 \begin{equation}   \label{formula:AB}    A = -\sqrt{{g\, (g-p)}\over1-p} \,,\,\, B = \sqrt{{(1-p)\, (g-p)}\over{g}},   a = {{A+1}\over{A-1}} \,,\,\, b =  {{B+1}\over{B-1}}.  \end{equation}
Expression (\ref{formula abAB}) follows by replacing $A$ and $B$ in $a, b$ in (\ref{formula:AB}).

 The second constitutive equation gives the value  $ \, d  = - c \, b $. From $(a+1)(c-1)= (b+1)c$ one deduces $c=-\frac{a+1}{b-a}$ and using $c-d=(b+1)c$ one deduces
$ \, d - c = {{(a+1)\, (b+1)}\over{b-a}} > 0 $, used in the final expression of $I_{\lambda}$.   This gives also the expressions $c=\frac{a+1}{a-b}, d= -\frac{b(a+1)}{a-b},
$ which can be expressed in terms of $p$ and $g$ through (\ref{formula abAB}). 
\end{proof}
\begin{figure}
    \centerline{\includegraphics[width = .8 \textwidth] {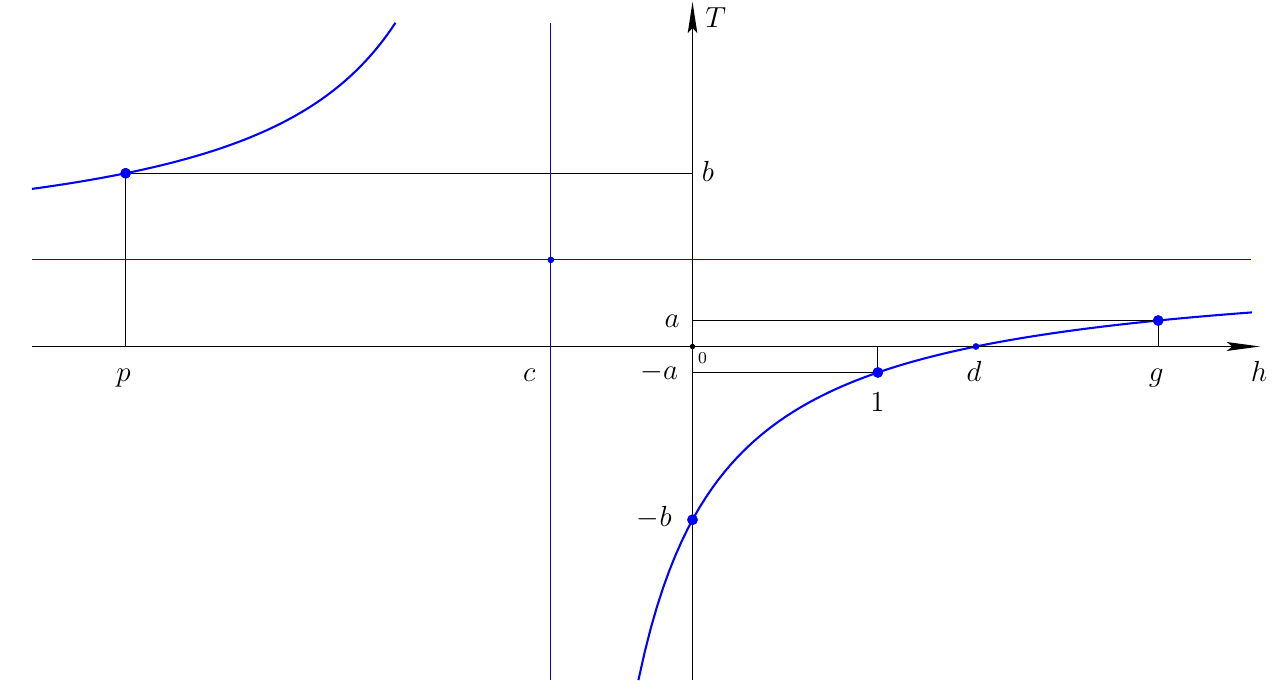}}
    \caption{Homographic transformation for the computation of the integral $\, \int_0^{\,1} \!\! {{\dd h}\over{\sqrt{\psi_\lambda (h)}}} $;  case $\, p < 0 < 1 < g $.}
    \label{fig-p-negatif-g-positif}
\end{figure}
If the sequence $ \, \sigma_0 ,\, \sigma_{1/2} ,\, \sigma_1 \, $ is convex, {\it id est} if
$ \,\,  \sigma_{1/2} - \sigma_0 < \sigma_1 - \sigma_{1/2} \,$, we have $ \, \delta < 0 $, $ \, a_0 < 0 \, $ and we are in the two other cases (thanks to the equivalence of Lemma \ref{P2-proposition-1}), that is two real roots with the same sign or two complex conjugate roots, given by
$\, p = {1\over2} \, (p+g -  \, \sqrt{\Delta}) \, $ and  $ \, g = {1\over2} \, (p+g+\, \sqrt{\Delta}) \, $ when $\, \Delta \, $ introduced in (\ref{discriminant})  is positive or $\, p = {1\over2} \, (p+g + i \, \sqrt{-\Delta}) \, $ and  $ \, g = {1\over2} \, (p+g- i \, \sqrt{-\Delta}) \, $ when $ \, \Delta < 0 \, $.

\begin{lemma} \label{P2-proposition-4}
  With the notations of Lemma \ref{P2-lemme-1}, in the case $\, \Delta > 0 \, $ with  two real roots with the same sign (that is necessarily $p,g>1$ or $p,g<0$),  one has
  
   $ \, \psi_\lambda (h) = |a_0| \, (h-c)^4 \, {{(a^2 - T^2)\, (b^2-T^2)}\over{(1-a^2)\,(b^2-1)}} \, $ positive for $ \, T^2 < a^2 < b^2 $.
   
For $p+g>2$, $b$ is given by (\ref{b-cas >2}), for $p+g<0$ $b$ is given by by (\ref{b-cas-p+g<0})  and $ \, a = b\, (p-g)+p+g-1 \, $ in both cases.  

\end{lemma}

\begin{figure}
    \centerline{\includegraphics[width = .8 \textwidth] {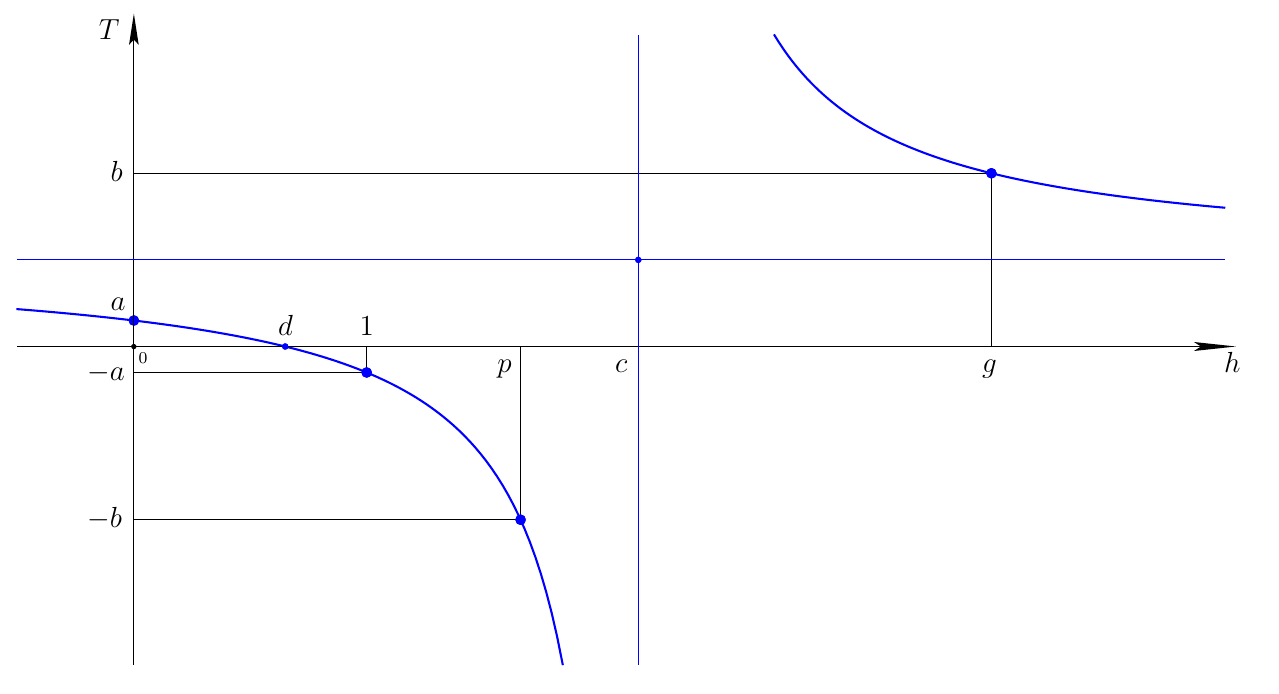}}
    \caption{Homographic transformation for the computation of the integral $\, \int_0^{\,1} \!\! {{\dd h}\over{\sqrt{\psi_\lambda (h)}}} $;  case $\,  0 < 1 < p < g $.}
    \label{fig-p-et-g-positifs}
\end{figure}

\begin{proof}
  (i) In the case $ \, p+g > 2 $, the homographic function, represented in  Figure \ref{fig-p-et-g-positifs}, satisfies
  
\begin{equation}
\label{hom-2}
    T(0) = a ,\, T(1) = -a ,\, T(p) = -b \,, \, T(g) = b \, . \end{equation}

  in order to enforce   the inequalities $ \,  0 < a < 1 < b \, $ and $ \, 0 < d < 1 < p < c < g $.
  
   \noindent

Identity (\ref{homographic}) yields the expression of $\psi_{\lambda}$ of Lemma \ref{P2-proposition-4}, positive for $ \, T^2 < a^2 < b^2 $.
  
The explicitation of the coefficients $ \, a \, , \, b \,, \, c \,,\, d \, $ is conducted as previously. We first have
$\left\{\begin{array}{l}c-d=(a-1)\, (0-c) \\
c-d= (-a-1)\, (1-c)\\
c-d=(-b-1)\,(p-c)\\
c-d=\,(b-1)\,(g-c) \end{array}\right.$

 The two first relations yield $ \, d = a \, c $, $ \, c =  {{a+1}\over{2 \,a}} $
  \noindent 
and we get after some algebra conducted with the help of formal calculus (the Sage package mentioned before), 
 $ \, a = b\, (p-g)+p+g-1 \, $ (using $d=ac$ and $c=\frac{a+1}{2a}$ in the third and fourth relations and taking the sum of the two resulting relations) and
the parameter $ \, b \, $ is a root of the second degree polynomial obtained by considering the last relation and replacing $a$ in terms of $b, p, g$ in $c$ then in $d$ to get an equation for $b$

 \begin{equation}\label{eq:b}{\check f}(b) \equiv (g-p)\,(g+p-1)\, (b^2+1) - 2 \, (p^2+g^2-p-g)\, b \end{equation}
which roots are of product 1. The reduced associated discriminant 

\noindent \centerline  {$
\Delta' = 4 \, p \, g \, (p-1)\, (g-1)  $} 

\noindent  is positive and we have just to compare the roots
with the value $ \, 1 \, $ because we want to enforce $ \, b > 1 $. But $ \,  {\check f}(1) = 4 \, p \, (1-p) < 0 \,$ and ${\check f}$ is convex, hence  $\ 1 \, $ is between the two roots, which yields

\begin{equation}
  b = {{1}\over{(g-p)\, (g+p-1)}} \big( p \,(p-1) +  g \,(g-1) + \sqrt{\Delta'} \big) > 1 \label{b-cas >2}\end{equation}
  
\noindent
and the first part of the proof is completed. 

(ii) In the case $ \,p+g < 0 $, we follow the same method.
The homographic function $ \, T(h) = {{h-d}\over{h-c}} \, $ is represented in  Figure \ref{fig-p-et-g-negatifs} and the parameters
satisfy the conditions $ \,  0 < a < 1 < b \, $ and $ \, p < c < g < 0 < d < 1 $.
We impose the following permutation between the roots:

\begin{equation} \label{hom-3}
  T(p) = b ,\, T(g) = -b ,\, T(0) = -a \, , \, T(1) = a \, .\end{equation}

 \noindent 
Identity (\ref{homographic}) yields the expression of $\psi_{\lambda}$ of Lemma \ref{P2-proposition-4}, positive for $ \, T^2 < a^2 < b^2 $. 
 The explicitation of the coefficients $ \, a $, $ \, b $, $ \, c \, $ and $ \, d \, $ comes from the conditions $ \, T(p) = b $, $\, T(g) = -b $, $\, T(0) = -a \, $ and $ \, T(1) = a $.
We deduce (using the same Sage package as mentioned there) that $ \, d = - a \, c $, $\, c =  -{{1-a}\over{2 \,a}} $, $ \, a =  b \, (p-g) + 1-p-g \, $ and the coefficient $ \, b \, $ is solution of the same equation ${\check f}(b)=0$.

 \noindent
 The two roots of this equation are positive,  we have $ \,  {\check f}(1) = 4 \, g \, (1-g) > 0 \, $ and ${\check f}$ is concave hence the root greater that 1 is

\begin{equation} \, b =  {{1}\over{(g-p)\, (1-p-g)}} \big( p \,(p-1) +  g \,(g-1) + \sqrt{\Delta'} \big) \, .\label{b-cas-p+g<0}\end{equation}

  \noindent
  The proof of the Lemma is completed. 
\end{proof}

\begin{figure}
    \centerline{\includegraphics[width = .8 \textwidth] {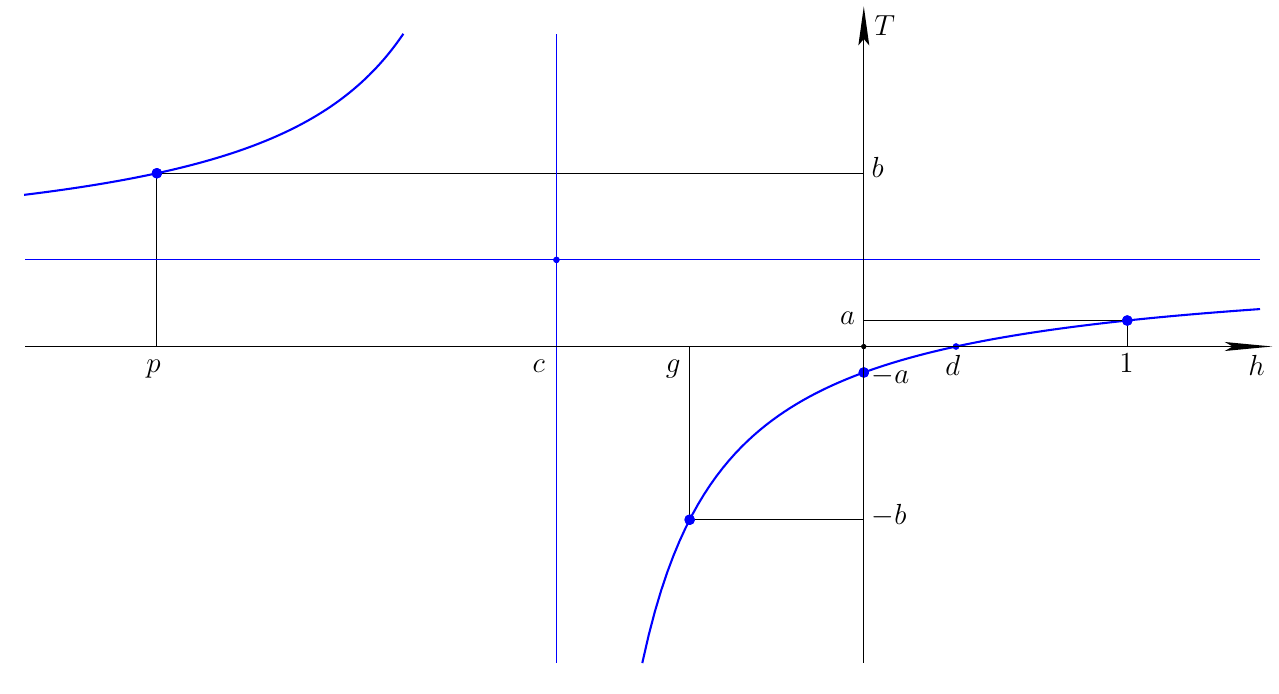}}
    \caption{Homographic transformation for the computation of the integral $\, \int_0^{\,1} \!\! {{\dd h}\over{\sqrt{\psi_\lambda (h)}}} $;  case $\,  p < g < 0 < 1 $.}
    \label{fig-p-et-g-negatifs}
\end{figure}

Treat finally the case of two complex conjugate roots.

\begin{lemma} \label{P2-proposition-3}
  With the notations of Lemma \ref{P2-lemme-1}, in the case $\, \Delta < 0 \, $ and $p+g\not=1$ with two conjugate roots, the expression of the homographic change of variable is split in two cases:

  (i) if $ \, p+g < 1 $, the constants $a, b, c, d$ are given by
  
     \centerline{ $c = - \frac{\sqrt{\widetilde{\Delta}} + {1\over4} \, ((p+g)^2 - \Delta)}{1-p-g} < 0 \,, \\ 
      a = {{1}\over{1-2\,   c}} \in (0,\, 1) \,,\,\, b =  {{\sqrt{-\Delta}}\over{2\, c - p-g}} \,,\,\,
      d = - a \,  c > c  \,, $}
      
      where  $\widetilde{\Delta} = {1\over16} \, ((p+g)^2 - \Delta) \, \big( (p+g-2)^2 - \Delta \big) > 0$.

  (ii) if $ \, p+g > 1 $, the relations of case (i) hold, except that we have now
  $ \, c > 0 $, $\, a < 0 \, $ and $ \,  d < c $. 

The expression of $\psi_{\lambda}$ reads $ \,  \psi_\lambda(h) = (-a_0) \, (h - c)^4 \, {{(a^2-(T(h))^2) \, ((T(h))^2 + b^2)}\over{(1-a^2)\, (1+b^2)}} \, , -a_0=|a_0|$ and the integral after the homographic change of variable is on $-a<T<a$.
\end{lemma}
\begin{proof}
The relations for the homographic change of variable are

\begin{equation} \label{hom-4}
T(0) = -a \, ,T(1) = a \,,  \, T \big( {{p+g \pm i \, \sqrt{-\Delta}}\over{2}} \big) = \mp i \, b\end{equation}

  \noindent 
with real coefficients $\, a $, $\, b $, $\, c \, $ and $ \, d $. Identity (\ref{homographic}) yields

$$ \,  \psi_\lambda(h) = (-a_0) \, (h - c)^4 \, {{(a^2-(T(h))^2) \, ((T(h))^2 + b^2)}\over{(1-a^2)\, (1+b^2)}} \, , -a_0=|a_0|.$$ The condition $\psi_{\lambda}>0$ for $h\in (0,1)$ implies that the integral is on $T\in (-a,a)$, thus imposes $\, a^2 < 1 $.  We construct the homography $ \, h \longmapsto T(h) $,
taking into consideration the necessary condition $ \, a^2 < 1 $. 
The constitutive relations $ \, T(0) = -a $, $ \, T(1) = a \, $ and $ \, T \big( {{\sigma \pm i \, \sqrt{-\Delta}}\over{2}} \big) = \mp i \, b \, $ 
take the form

{\small $
  c - d = (-a-1) \, (0-c) = (a-1) \, (1-c) = (-i\, b -1) \, ( {{\sigma + i \, \sqrt{-\Delta}}\over{2}} -c )   = (i\, b -1) \, ( {{\sigma - i \, \sqrt{-\Delta}}\over{2}} -c ) \, .$}

\noindent
Then $ \, {{i \, b +1}\over{i \, b -1}} = {{i \, \sqrt{-\Delta} + (2 \, c - (p+g))}\over{i \, \sqrt{-\Delta} - (2 \, c - (p+g))}} $
and $ \, b = {{\sqrt{-\Delta}}\over{2\, c - (p+g)}} $. We have also $ \, {{a+1}\over{a-1}} = {{1}\over{c}} - 1 \, $ then
$ \,  a = {{1}\over{1-2\, c}} $. 

We  can write now the equation relative to the parameter $ \, c $:
$\, c \, (a+1) = (i\, b +1 ) \, ( c - {{(p+g) + i \, \sqrt{-\Delta}}\over{2}} ) $.
But $ \, (i\, b +1 ) \, ( c - {{(p+g) + i \, \sqrt{-\Delta}}\over{2}} ) = {{(2\, c-(p+g))^2 - \Delta}\over{2 \, (2\, c -(p+g))}} \, $
and $ \, a + 1 = 2\, {{1-c}\over{1 - 2\, c}} $. An equation of degree 2 for the parameter $ \, c \, $ is emerging:

\noindent
\begin{equation}\label{eq:c}
 f(c) \equiv (1-(p+g))\, c^2 + 2 \, {{(p+g)^2 - \Delta}\over{4}} \, c - {{(p+g)^2 - \Delta}\over{4}} = 0 \, . \end{equation}

\noindent
We remark that the constant term $ \, {{(p+g)^2 - \Delta}\over{4}} \, $ is $ \, p \,  \, g $, that is $\vert p\vert^2>0$. The function $f$ have thus two roots $\, c_- \, $ and $ \, c_+ \, $ solution of $\vert p\vert^2(\frac{1}{c}-1)^2= 1-(p+g)+\vert p\vert^2>0$.

We distinguish now the two cases for $p+g$:

(i) If  $ \, p+g < 1 $, we remark that $f$ is convex and that $ \, f({{(p+g)}\over{2}}) = -{1\over4} \,\Delta \, ((p+g)-1) < 0 \, $
in this case and $ \, {{p+g}\over{2}} \, $ is in $(c_-, c_+)$. We have also $ \,  f({{1}\over{2}}) = {1\over4}\, (1-(p+g)) > 0 \, $ 
and $ \, {1\over2} \, $ is not in $(c_-, c_+)$. We deduce the inequalities $ \, c_- < 0 $, $\, c_- <  {{(p+g)}\over{2}} < c_+ <  {1\over2} $, $ \, c_+ > 0 $. 
Note that the condition $a^2<1$ yields either $0<a<1$ or $-1<a<0$. Before making a choice between the two real roots $ \, c_\pm $, observe that we have
$ \, c - d = (a+1) \, c \, $ and $ \, d = - a \, c $. Moreover $ \, a = {{1}\over{1 - 2 \, c}} \, $ . If $ \,a > 0 $, the condition
$ \, a^2 < 1 \, $ take the form $ \, a - 1 < 0 $,  but $ \, a-1 = {{2\, c}\over{1 - 2 \, c}} $ and we must have $ \, c<0 $. Then $ \, d > 0 \, $ and
$ \, c < 0 < d \,$; the condition $ \, d-c>0 \, $ is satisfied. If we had chosen $-1<a<0$, $a=\frac{1}{1-2c}$ and $a+1=\frac{2(1-c)}{1-2c}$ hence $1-2c<0$ and $1-c<0$, hence $c>1$ and one would have $0<d<c$.

The reduced discriminant $\,  \widetilde{\Delta} \, $
is equal to $ \, \big( {{(p+g)^2 - \Delta}\over{4}} \big)^2 + (1-(p+g)) \, \big( {{(p+g)^2 - \Delta}\over{4}} \big) = 
{1\over16} \, ((p+g)^2 - \Delta) \, \big( ((p+g)-2)^2 - \Delta \big) \, $ and is positive. Then
$ \, c =  c_- =  {{-1}\over{1-(p+g)}} \, \big( \sqrt{\widetilde{\Delta}} + {1\over4} \, ((p+g)^2 - \Delta) \big) <0$. We have $0<a<1$

(ii) If  $ \, p+g > 1 $, the two roots $ \, c_\pm \, $ of $f(c)=0$, which is concave, have the same sign. 
 We have now   $ \,f({{\sigma}\over{2}}) = {1\over4} \,\Delta \, (1-(p+g)) > 0 \, $ 
 and $ \, {{\sigma}\over{2}} \, $ is still between the two roots.
Moreover,  $ \, f({{1}\over{2}}) = -{1\over4}\, (\sigma-1) < 0 \, $ 
and $ \, {1\over2} \, $ is outside the two roots. We deduce the inequalities
$ \, {1\over2} < c_- <  {{\sigma}\over{2}} < c_+  $. Then whatever the choice between $\, c_\pm $, $\, 2\, c-1 > 0 \, $ and $\, a < 0 $. 
The condition $ \, a^2 < 1 \, $ imposes now $\, a + 1 > 0 \, $ and $ \, a+1 = 2\, {{1-c}\over{1 - 2\, c}} \, $ is positive if and only of $ \, 1-c < 0 $. 
We observe also that $ \,  f(1) = \big(1- {{(p+g)}\over{2}}\big)^2 - {{\Delta}\over4}  > 0  $, then $ \, c_- < 1 < c_+ $. 
The choice $ \, c_+ \, $ is mandatory and we have \begin{equation}\label{valeur-c} \, c = c_+ =  {{1}\over{(p+g)-1}} \, \big( \sqrt{\widetilde{\Delta}} + {1\over4} \, ((p+g)^2 - \Delta) \big) \end{equation}  Note that $c$ has the same expression in the two cases. Lemma \ref{P2-proposition-3} is proven.
\end{proof}
This ends the proof of Proposition \ref{homographic-change}. 
\end{proof}


\bigskip \noindent $\bullet$ 
Third case:  $ \, \Sigma \, $ continuous  positive fonction, affine in  $ \, (0,\, {1\over2} ) \, $ as well as in $ \, ( {1\over2} ,\, 1 ) $.

The function $ \, \Sigma \, $ is defined by its values  $\, \Sigma(0)=\sigma_0 $, $ \, \Sigma(\frac12)=\sigma_{1/2} \, $  and $ \, \Sigma(1)=\sigma_1 $.
We introduce new parameters, still denoted by $\, \alpha \, $ and $ \, \beta \, $ to represent the data:

\centerline  {$
  \sigma_0 = \sigma_{1/2} \, (1-\alpha) \,$ and $ \, \sigma_1 = \sigma_{1/2} \, (1+\beta) \, .$ }

\noindent
Then the inequalities $ \, \alpha < 1 \, $ and $ \, \beta > -1 \, $ express
the constraints $ \, \sigma_0 > 0 \, $ and $ \,  \sigma_1 > 0 $. Moreover, $ \,  \sigma_{1/2} \, $ is positive. 

\begin{lemma} \label{expression-psi-p1-par-morceaux}
  If the continuous function $ \, \Sigma \, $ defined on $ \, (0,\, 1) \, $ by its values

  \centerline { $
    \Sigma(0)=\sigma_0 $, $ \, \Sigma(\frac12)=\sigma_{1/2} $, $ \, \Sigma(1)=\sigma_1 $} 

    \noindent is affine in each interval $ \, (0,\, {1\over2} ) \, $ and $ \, ( {1\over2} ,\, 1 ) $, the function $\, \psi_\lambda \, $
 given by (\ref{psilambda})
  \noindent admits the following expression
  
   \centerline   {$
    \psi_\lambda(h) =  \left \{ \begin {array}{l}  \psi^0(h)  \quad  {\rm if} \, h \leq {1\over2} \\
    \psi^1(h)  \qquad  {\rm if} \, h \geq {1\over2} \,, 
    \end{array} \right.  $}

   \noindent with
  
 \begin{equation} \label{psi0-psi1-fev2025}  \left \{ \begin {array}{l}
       \psi^0(h) = -{2\over3}\, \alpha \, \zeta \, h^3 + (\alpha\, \zeta - \zeta + 1) \, h^2 +
      \big( {{\beta-5\, \alpha}\over{12}}\, \xi + \xi -1 \big) \, h \\ 
    \psi^1(h) = {2\over3}\, \beta \, \zeta \, (1-h)^3 + (-\beta \, \zeta - \zeta + 1) \, (1-h)^2 +
      \big( {{5 \, \beta - \alpha}\over{12}}\, \zeta + \zeta -1 \big) \, (1-h)  \,, 
    \end{array} \right.  \end{equation}

   \noindent    with $\, \zeta \equiv \lambda \, \sigma_{1/2} $.
   We have in particular $\,  \psi_\lambda ({1\over2}) =  {{\beta-\alpha}\over{24}}\, \xi + {1\over4}\, (\xi-1) $.
   Moreover, if $\,  \psi_\lambda \, $ is positive on the interval $ \, (0,\, 1) $, 

  \centerline { $
    {{\dd \psi_\lambda}\over{\dd h}} (0) =  {{\beta-5\, \alpha}\over{12}}\, \xi + \xi -1 > 0 \,,\,\,
     {{\dd \psi_\lambda}\over{\dd h}} (1) = - \big( {{5 \, \beta - \alpha}\over{12}}\, \xi + \xi -1 \big) < 0 $.}

\end{lemma}
\begin{remark} \label{u-et-un-moins-u}
   We observe that the expression of $\, \psi^1 \, $ is obtained from the expression of $ \,  \psi^0 \, $ by making the transformations 
   $\, \alpha \longleftrightarrow (-\beta) \, $ and $ \, h  \longleftrightarrow  (1-h) $. 
\end{remark} 
  
   \begin{proof}
     The function  $ \, \Sigma \, $ admits the algebraic expression: 
  
   \centerline   {$
    \Sigma(h) =  \left \{ \begin {array}{l} (2\, \alpha \, h - \alpha +1) \, \sigma_{1/2} \,\,\,\, {\rm if} \, h \leq {1\over2} \\
      (2\, \beta \, h - \beta +1) \, \sigma_{1/2} \,\,\,\,\, {\rm if} \, h \geq {1\over2} \,. 
    \end{array} \right.  $}

   \noindent
   We integrate twice, enforce the conditions  $\,  \psi_\lambda(0) =  \psi_\lambda(1) = 0 \, $ and impose the continuity
   of $\,  \psi_\lambda \, $ and  $\,  {{\dd \psi_\lambda}\over{\dd h}}  \, $ at the specific value $ \,  h = {1\over2} $. The result follows.    
\end{proof}

   We have to compute the integral $ \, I_{\lambda} \equiv \int_0^{\,1}\!\!  {{\dd h}\over{\sqrt{\psi_\lambda (h)}}} $. We have the following calculus:

\noindent \centerline  {$
I_{\lambda} = \int_0^{\,1/2}\!\!\!  {{\dd h}\over{\sqrt{\psi^0 (h)}}} +  \int_{1/2}^{\,1} \!  {{\dd h}\over{\sqrt{\psi^1 (h)}}}
=  \int_0^{\,1/2} \! {{\dd h}\over{\sqrt{\psi^0 (h)}}} +  \int_0^{\,1/2} \!  {{\dd h}\over{\sqrt{\psi^1 (1-h)}}}   $.}
   
\noindent

\begin{figure}
  \centerline{{\includegraphics[width = .33 \textwidth] {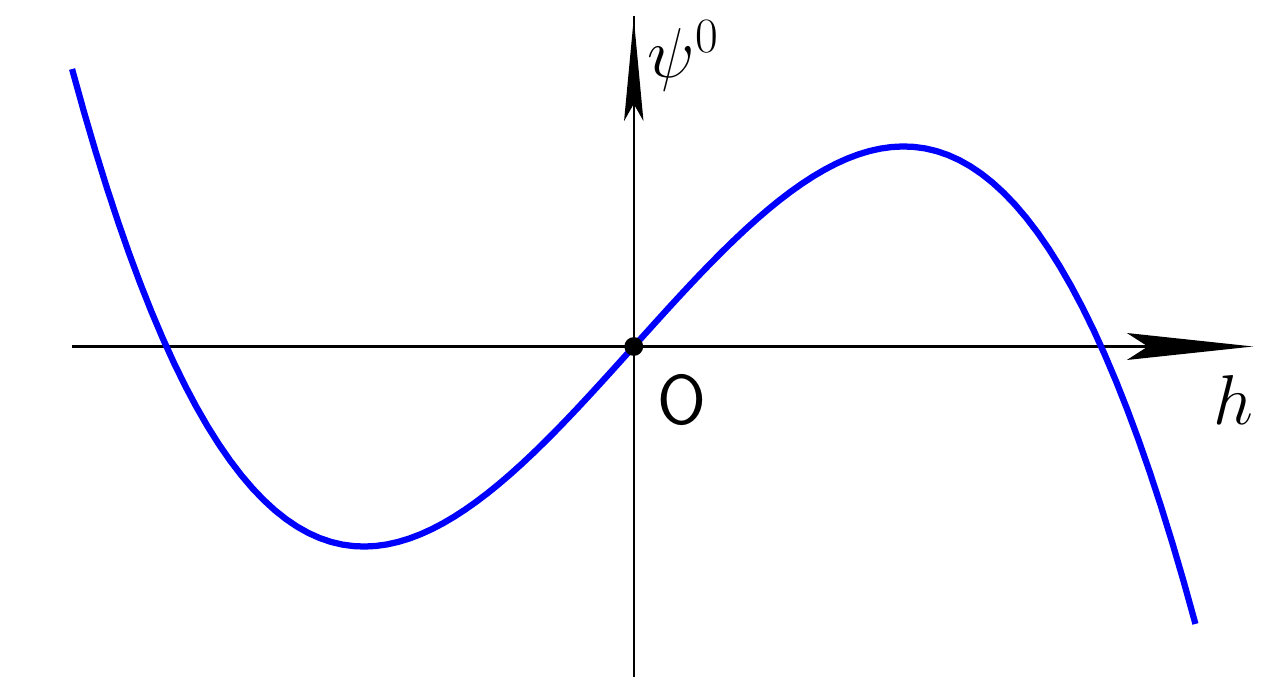}}
    {\includegraphics[width = .33 \textwidth] {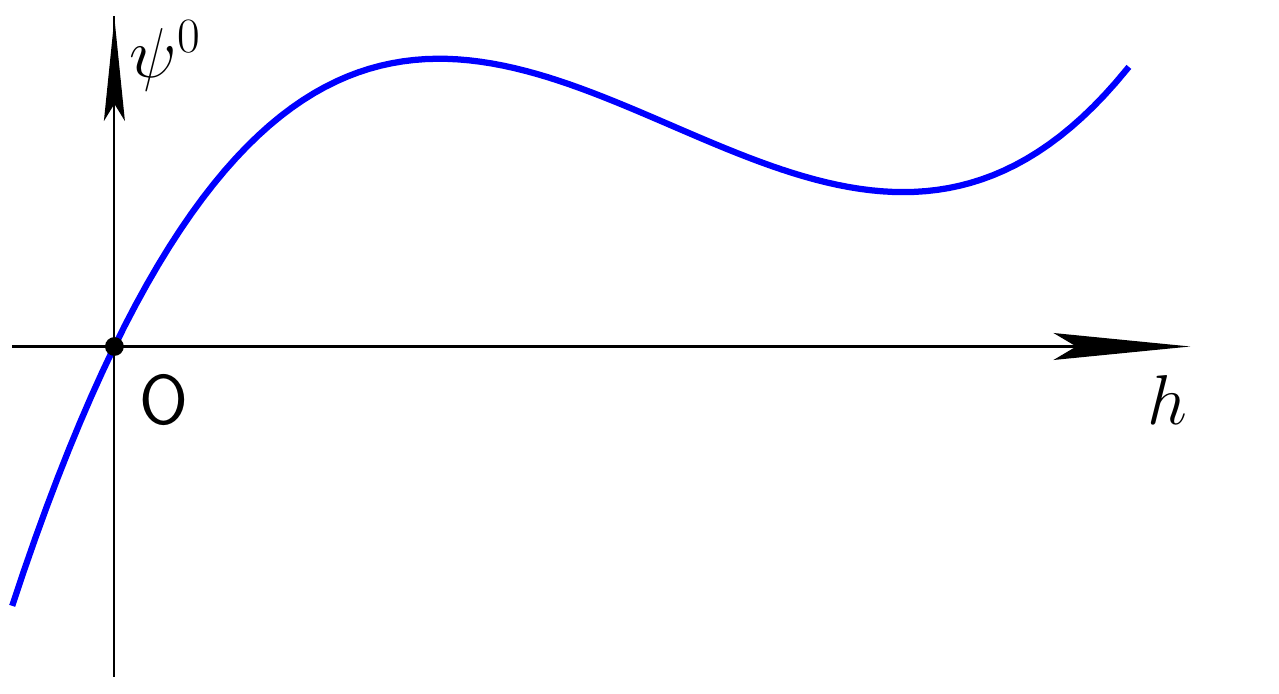}} {\includegraphics[width = .33 \textwidth] {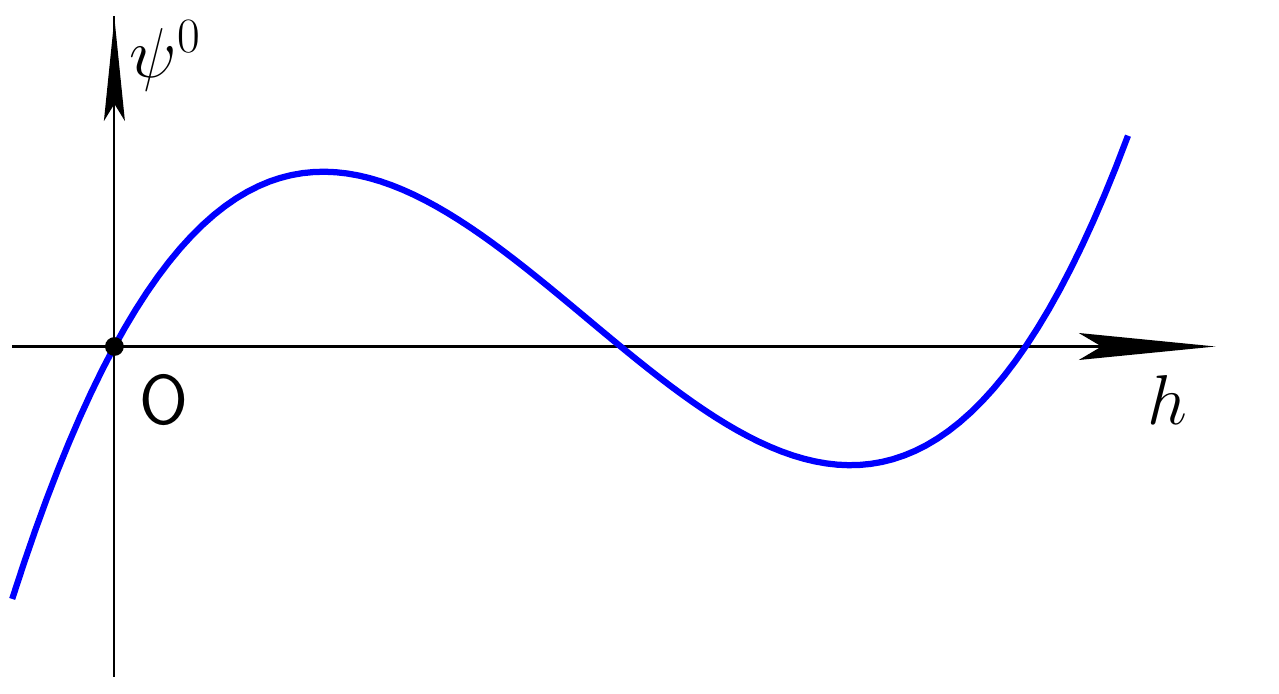}}}
  \caption{Contribution $\, \psi^0 \, $ to the function $\, \psi_\lambda $. The case $ \, \alpha>0 \, $ is on the left and the case $ \, \alpha < 0 \, $
  is splitted into two figures: two complex roots in the middle and three real roots on the right.}
    \label{psi-zero}
\end{figure}

\begin{figure}
  \centerline{{\includegraphics[width = .33 \textwidth] {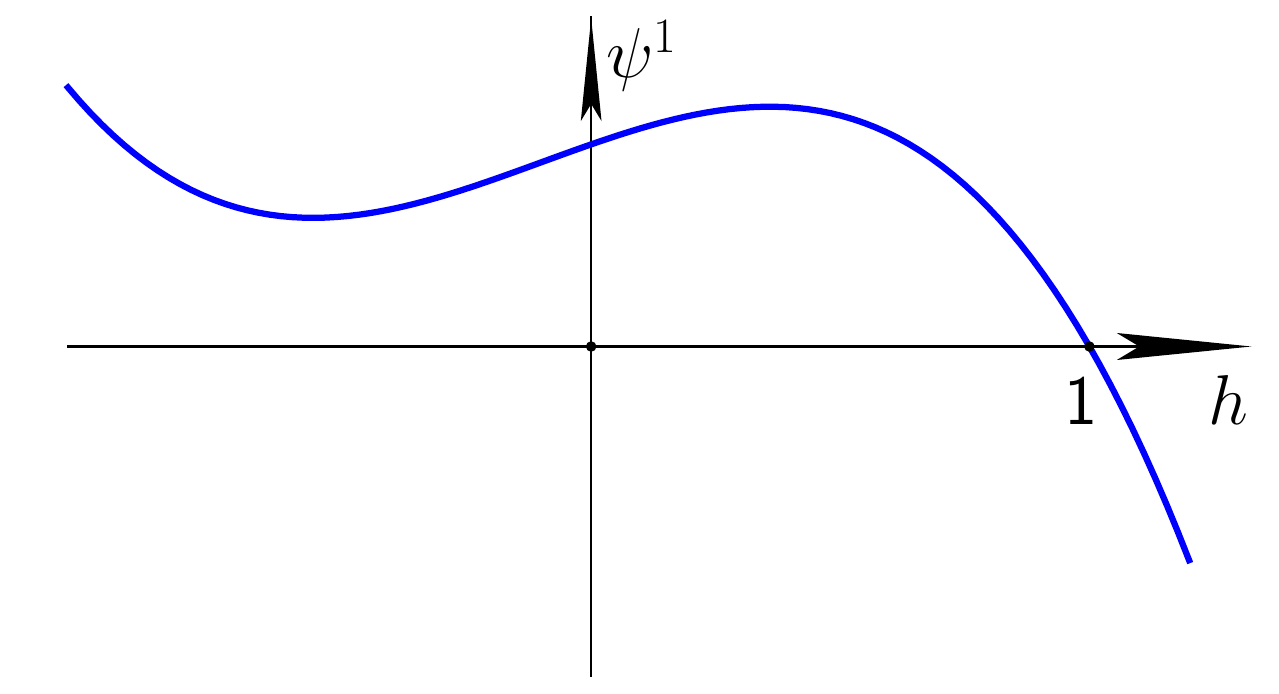}}
    {\includegraphics[width = .33 \textwidth] {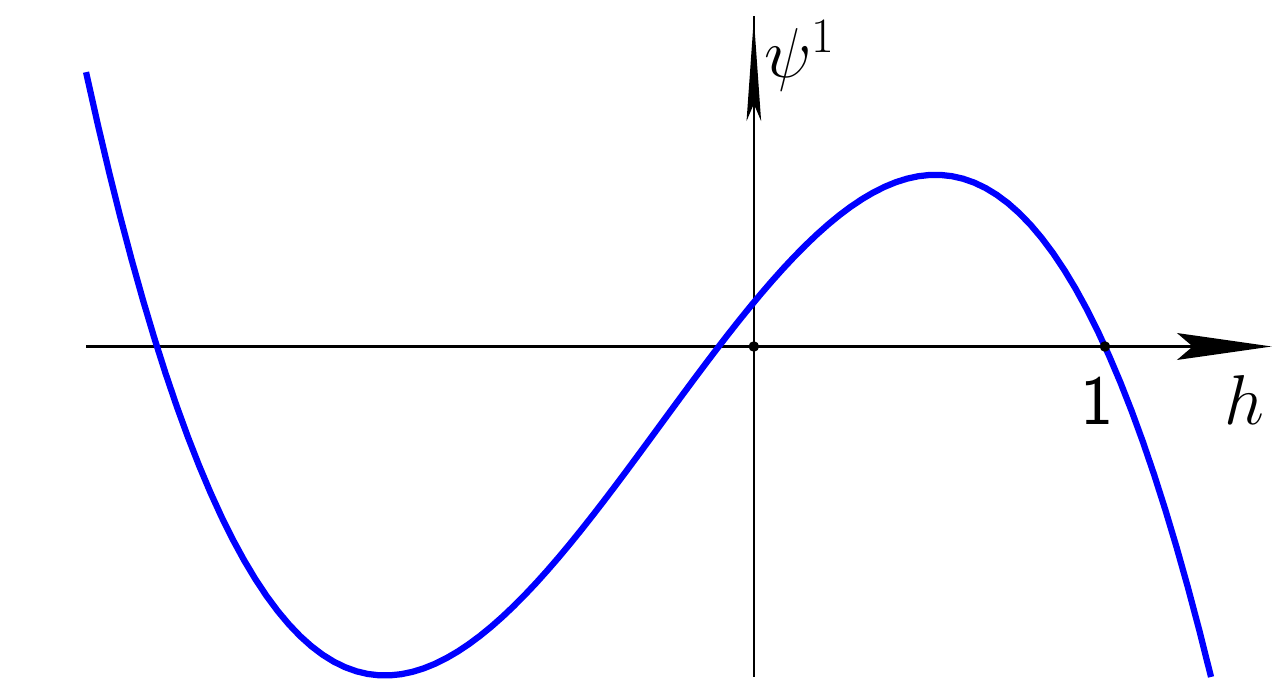}} {\includegraphics[width = .33 \textwidth] {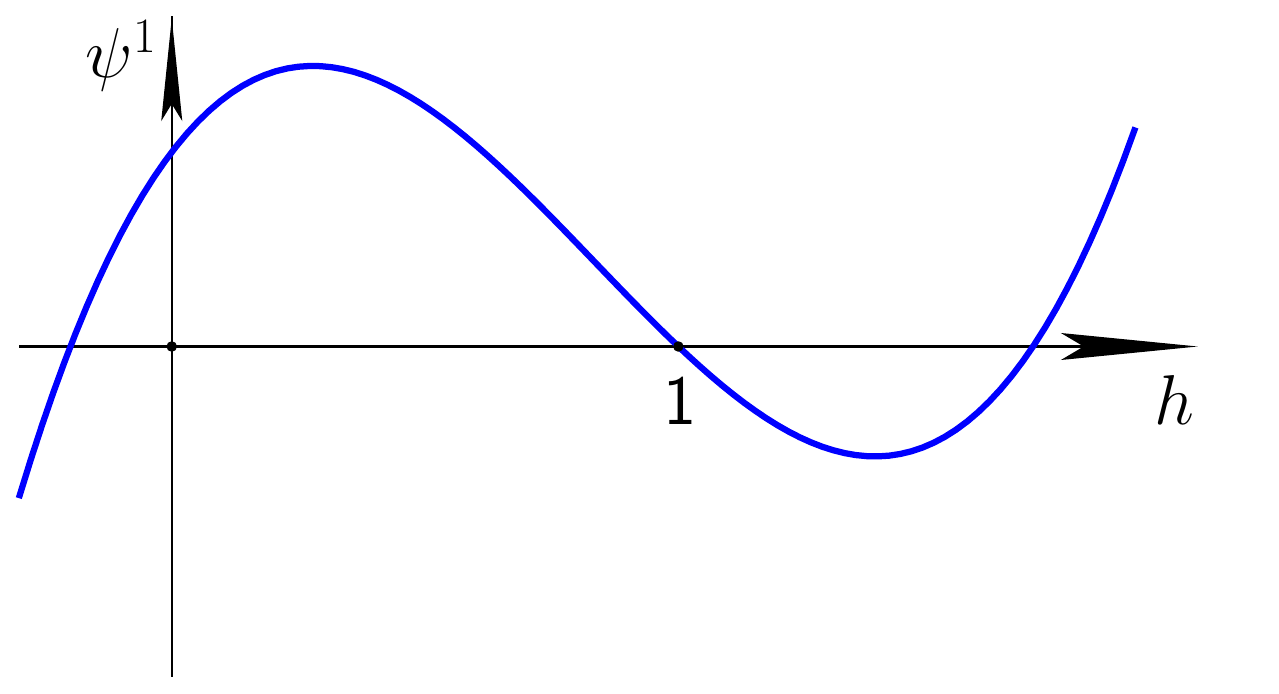}}}
  \caption{Contribution $\, \psi^1 \, $ to the function $\, \psi_\lambda $. The case $ \, \beta>0 \, $ is splitted into two sub-cases: 
    two complex roots on the left and three real roots in the middle; the  case $ \, \beta<0 \, $ is presented     on the right.}
    \label{psi-un}
\end{figure}

\noindent 
Due to Remark  \ref{u-et-un-moins-u}, the determination of the second term relative to $ \, \psi^1 \, $ is very analogous to the term associated to $ \, \psi^0 $.
In the following, we will concentrate essentially to the evaluation of the first integral $ \, I_0 \equiv  \int_0^{\,1/2}\!\!\!  {{\dd h}\over{\sqrt{\psi^0 (h)}}} $. 
   
\begin{proposition} \label{P1-par-morceaux-alpha-0}
  If $ \, \alpha = 0 $, the integral $ \, I_0 \equiv \int_0^{\,1/2} \!\!  {{\dd h}\over{\sqrt{\psi^0 (h)}}} \, $
  admits several  expressions parameterized by $ \, \xi \equiv \lambda \, \sigma_{1/2} $:

(i) if $\, \xi = 1 $, then $ \, \beta > 0 \, $ and $ \, I_0 = 2 \, \sqrt{{6}\over{\beta}} \, $,  
  
(ii) if $\, \xi > 1 $, 
  $\, I_0 = {{2}\over{\sqrt{\xi-1}}} \,  \arctan \big( {{1}\over{\varphi_0}} \big) \, $ with 
$ \,  \varphi_0  = \sqrt{1 + {{\beta\, \xi}\over{6\, (\xi-1)}}} \,\,\, $ and  

(iii) if $\, \xi <  1 $,    $  \, I_0 = {{1}\over{c \sqrt{b_0}}} \, \log {{\sqrt{c^2 + 2} + c}\over{\sqrt{c^2 + 2} - c}} \, $
with $ \, b_0 =  {{\beta}\over{12}} \, \xi + \xi -1 > 0 \, $ and $ \, c = \sqrt{{1-\xi}\over{b_0}} $. 
\end{proposition}
\begin{proof}
  We have in this particular case $ \, \psi^0 (h) = -(\xi-1) \, h^2 + \big( {{\beta}\over{12}} \, \xi + \xi -1 \big) \, h $.

  (i) If $\, \xi = 1 $,  $ \,  \psi^0 (h) =  {{\beta}\over{12}} \, h \, $ and $ \, I_0 = \sqrt{{12}\over{\beta}} \,  \int_0^{\,1/2}\!  {{\dd h}\over{\sqrt{h}}}
  =  \sqrt{{24}\over{\beta}} $.
  
  (ii) If $\, \xi > 1 $, we write $ \,  \psi^0 (h) = (\xi-1) \, h \, \Big( \big( 1 + {{\beta\, \xi}\over{12\, (\xi-1)}} \big) -h  \Big) $.
  With $\, a \equiv 1 + {{\beta\, \xi}\over{12\, (\xi-1)}} > 0 $, we make the classical change of variable
  $\, h \, (a-u) = \varphi^2 \, h^2 $. Then $ \, \varphi^2 = {{a}\over{h}} -1 \, $ and $ \, 2 \, \varphi \, \dd \varphi = - {{a}\over{h^2}} \, \dd h $.
If $ \, h={1\over2} $, then $ \, \varphi = \varphi_0 > 0 \, $ with $\, \varphi_0^2 = 2\, a -1 = 1 + {{\beta\, \xi}\over{6\, (\xi-1)}} $.   
  In consequence,
  
  \noindent \qquad $ 
  I_0 = \int_{+\infty}^{\varphi_0} \Big( {{-2 \, h^2 \, \varphi \, \dd \varphi}\over{a \, \sqrt{\xi-1}}} \Big) \, {{1}\over{\varphi \, h}} $
$      = {{2}\over{\sqrt{\xi-1}}} \, \int_{\varphi_0}^{+\infty} \! {{\dd \varphi}\over{1+\varphi^2}} $

  \noindent \qquad \quad $ 
     = {{2}\over{\sqrt{\xi-1}}} \, \big({{\pi}\over{2}} -  \arctan (\varphi_0) \big)  $
$ =  {{2}\over{\sqrt{\xi-1}}} \,  \arctan \big( {{1}\over{\varphi_0}} \big) $

\noindent because $ \,  \varphi_0 > 0 $. 
  
(iii) If $\, \xi <  1 $, we can write $ \,  \psi^0 (h) = b_0 \, ( h + c^2 \, h^2 ) \, $ with
$ \, b_0 =  {{\beta}\over{12}} \, \xi + \xi -1 > 0 \, $ and $ \, c = \sqrt{{1-\xi}\over{b_0}} $.
We make now the change of variable $ \,  h + c^2 \, h^2 = \theta^2 \, h^2 $. Then $ \, \theta^2 = c^2 + {{1}\over{h}} > c^2 \, $ and 
$ \, 2 \, \theta \, \dd \theta = - {{1}\over{h^2}} \, \dd h $. Then

  \noindent \qquad $ 
  I_0 = {{1}\over{\sqrt{b_0}}} \int_{+\infty}^{\sqrt{c^2+2}} (-2 \, \theta \, h^2 \, \dd \theta ) \, {{1}\over{\theta \, h}} $
$ = {{2}\over{\sqrt{b_0}}} \int_{\sqrt{c^2+2}}^{+\infty} \, h \, \dd \theta =  {{2}\over{\sqrt{b_0}}} \int_{\sqrt{c^2+2}}^{+\infty} \, {{\dd \theta}\over{\theta^2 - c^2}} $

  \noindent \qquad  \quad $ 
  =  {{1}\over{c \sqrt{b_0}}} \Big[ \log {{\theta - c}\over{\theta+c}} \Big]_{\sqrt{c^2+2}}^{+\infty}    $
  $ =  {{1}\over{c \sqrt{b_0}}} \, \log {{\sqrt{c^2 + 2} + c}\over{\sqrt{c^2 + 2} - c}} $. 

  \noindent
  The reader will observe that if $ \, \xi \, $ tends to 1, each of the results proposed in (ii) and (iii) converge towards the expression proposed in (i). 
\end{proof}

\begin{proposition} \label{P1-par-morceaux-alpha-strictement-positif}
  If $ \, \alpha > 0$, recall that
  $ \, \psi^0(h) = -{2\over3}\, \alpha \, \xi \, h^3 + (\alpha\, \xi - \xi + 1) \, h^2 + \big( {{\beta-5\, \alpha}\over{12}}\, \xi + \xi -1 \big) \, h $.
  It corresponds to the left part of Figure \ref{psi-zero}. The function $ \, \psi^0 \, $ admits three real roots and we set
  $ \,  \psi^0(h) \equiv a_0 \, (h-r_-) \, h \, (r_+-h) \, $ with $ \, a_0 = {2\over3}\, \alpha \, \xi > 0 \, $ and $ \, r_- < 0 < r_+ $.
  Then $ \, I_0 \equiv \int_0^{\,1/2} \!\!  {{\dd h}\over{\sqrt{\psi^0 (h)}}} = {{2}\over{\sqrt{a_0 \, (-r_-)}}} \, K(\varphi_0 ,\, m)  \, $ with
  $ \, m = {{r_+}\over{r_-}} < 0 \, $ and $ \, \smash{ \varphi_0 = \arcsin \big( {{1}\over{\sqrt{2 \, r_+}}} \big)} $.
\end{proposition} 
\begin{proof} 
  The explicitation of the algebraic expressions of $ \, r_- \, $ and $ \, r_+ \, $ has no interest and is not detailed here. 
  To compute $ \, I_0 $, we make the change of variables $ \, h = r_+ \, \sin^2 \theta $. Then $ \, \dd h = 2 \, r_+ \, \sin \theta \, \cos \theta $.
  Then
  
  \noindent \qquad $ 
  \psi^0(h) \, = a_0 \, (r_+ \, \sin^2 \theta - r_-) \,  r_+ \, \sin^2 \theta \, r_+ \, \cos^2 \theta $
  
  \noindent \qquad  \qquad  \quad $ 
= a_0 \, r_+^2 \,  \sin^2 \theta \,  \cos^2 \theta \, (-r_-) \, \big( 1 -  {{r_+}\over{r_-}}\, \sin^2 \theta \big) $

  \noindent \qquad  \qquad  \quad $ 
 = a_0  \, (-r_-) \,  r_+^2 \,  \sin^2 \theta \,  \cos^2 \theta \, (1-m \, \sin^2 \theta)     $ 

  \noindent
  with $ \,  m = {{r_+}\over{r_-}} < 0 $. 
  Moreover, if $ \, h = {1\over2} $, then $\, \theta = \varphi_0 \, $ with $ \, \sin^2 \varphi_0 = {{1}\over{2 \, r_+}} $.
  Then
  $ \, I_0 = \int_0^{\,\varphi_0} \! {{2 \, \dd \theta}\over{\sqrt{a_0 \, (-r_-)}}} \, {{1}\over{\sqrt{1 -   m \, \sin^2 \theta}}} \,   $  
and the result is established.   
\end{proof}

\begin{proposition} \label{P1-par-morceaux-alpha-strictement-negatif}
  The case  $ \, \alpha < 0$, corresponds to the middle and right pictures of  Figure \ref{psi-zero}.
  Recall that 
  $ \, \psi^0(h) = -{2\over3}\, \alpha \, \xi \, h^3 + (\alpha\, \xi - \xi + 1) \, h^2 + \big( {{\beta-5\, \alpha}\over{12}}\, \xi + \xi -1 \big) \, h $.

  (i) If the polynomial  $ \, \psi^0(h) \, $ has two complex roots, we write it under the form
  $ \,  \psi^0(h) \equiv a_0 \, h \, (h^2 + \mu \, h + \zeta^4) \, $ with  $ \, a_0 = {2\over3}\, |\alpha| \, \xi > 0 \, $ and we have the inequality
  $ \, \mu^2 - 4 \, \zeta^4 < 0 $; we introduce $ \, m = {1\over2} - {{\mu}\over{4 \, \zeta^2}} \, $ that satisfies $ \, 0 < m < 1 \, $ and 
  $ \, \varphi_0 = 2 \, \arctan \big( {{1}\over{\zeta  \sqrt{2}}} \big) $. Then we have
  $ \, I_0 = {{1}\over{\zeta  \sqrt{a_0}}} \, K(\varphi_0 ,\, m) $.

  (ii) If the polynomial  $ \, \psi^0(h) \, $ has three real roots, we set
  $ \,  \psi^0(h) \equiv a_0 \, h \, (r_- - h) \, (r_+-h) \, $ with  $ \, a_0 = {2\over3}\, |\alpha| \, \xi > 0 \, $ and $ \, 0 < r_- < r_+ $;
  we consider $ \, m = {{r_-}\over{r_+}} \in (0,\, 1) \, $ and $ \,  \varphi_0 = \arcsin \big( {{1}\over{\sqrt{2 \, r_-}}} \big) $.
The integral $ \, I_0 \, $ is  computed with the following expression: $ \, I_0 = {{2}\over{\sqrt{a_0 \, r_+}}} \,  K(\varphi_0 ,\, m) $.
\end{proposition}
\begin{proof}
  This  proof is directly inspired by the book \cite{AS} (paragraph 17.4.61 and 17.4.70). Nevertheless, we give the details herein for a complete explanation of the
  final relations. 
  
  (i) We operate the  change of variable $ \, h = \zeta^2 \, t^2 \, $ with $ \, t = \tan {{\theta}\over{2}} $. Then

  \noindent \qquad $ 
  \psi^0(h) \, = a_0 \, \zeta^2 \, t^2 \, \big( \zeta^4 \, (1+t^4) + \mu \, \zeta^2 \, t^2 \big) $

  \noindent \qquad  \qquad  \quad $ 
= a_0 \, \zeta^2 \, t^2 \, \big( \zeta^4 \, (1+t^4) + \mu \, \zeta^2 \, t^2 \big) $

  \noindent \qquad  \qquad  \quad $ 
= a_0 \, \zeta^6 \, t^2 \, \big( (1+t^2)^2 - 2\, t^2 + {{\mu}\over{\zeta^2}}  \, t^2 \big) $

  \noindent \qquad  \qquad  \quad $ 
  = a_0 \, \zeta^6 \, t^2 \,(1+t^2)^2 \, \big[ 1- \big( {1\over2} - {{\mu}\over{4 \, \zeta^2}} \big) \, {{4 \, t^2}\over{(1+t^2)^2}} \big] $.

  \noindent 
  Because the discriminant  $ \, \mu^2 - 4 \, \zeta^4 < 0 $ is negative, we have

  \centerline { $
     {{|\mu|}\over{2 \, \zeta^2}} < 1 \, $
     and $ \, -1 < -  {{\mu}\over{2 \, \zeta^2}} < 1 $.}

  \noindent 
    In other terms, $ \,  -{1\over2} <  - {{\mu}\over{4 \, \zeta^2}} < {1\over2} \, $ and $ \, 0 < m \equiv  {1\over2} - {{\mu}\over{4 \, \zeta^2}} < 1 $. 
    We remark that $ \, \sin \theta = {{2 \, t}\over{1+t^2}} \, $ and we get
    $ \,  \psi^0(h)  = a_0 \, \zeta^6 \, t^2 \,(1+t^2)^2 \, ( 1 - m \, \sin^2\theta) $. We have also $ \, \dd h = 2 \, \zeta^2 \, t \, (1+t^2) \, {{\dd \theta}\over{2}} \, $ 
    and  $ \, \dd h = \zeta^2 \, t \, (1+t^2) \, \dd \theta $. With this change of variables, the upper bound is equal to $ \, \varphi_0 \, $ such that
    $ \, \zeta^2 \, \tan^2{{\varphi_0}\over{2}} = {1\over2} $, {\it id est} $\, \varphi_0 = 2 \, \arctan \big( {{1}\over{\zeta \sqrt{2}}} \big) $. We can now achieve the calculus
    of the integral:
    
  \noindent \qquad $ 
  I_0 =  \int_0^{\,1/2} \!\!  {{\dd h}\over{\sqrt{\psi^0 (h)}}} $
$     = \int_0^{\,\varphi_0} \! {{\zeta^2 \, t \, (1+t^2) \, \dd \theta}\over{ \sqrt{a_0} \, \zeta^3 \, t \,(1+t^2) \, \sqrt{ 1 - m \, \sin^2\theta }}} $

  \noindent \qquad \quad $ 
    =  {{1}\over{\zeta  \sqrt{a_0}}} \,  \int_0^{\,\varphi_0} \! {{\dd \theta}\over{\sqrt{ 1 - m \, \sin^2\theta }}} =  {{1}\over{\zeta  \sqrt{a_0}}} \, K(\varphi_0 ,\, m) $.

    (ii) If the polynomial  $ \, \psi^0(h)  \equiv a_0 \, h \, (r_- - h) \, (r_+-h) \, $  has three real roots, we set
    $ \, h = r_- \, \sin^2 \theta \, $ and

    \centerline {$ 
  \psi^0(h)  = a_0 \,\, r_- \, \sin^2 \theta \,\,  r_- \, \cos^2 \theta \, (r_+ - r_- \, \sin^2 \theta)
  = a_0 \,\,r_-^2 \, \sin^2 \theta \, \cos^2 \theta \, \big( 1 - {{r_-}\over{t_+}} \, \sin^2 \theta \big) $.}

  \noindent 
    With $\, 0 < m \equiv {{r_-}\over{t_+}} < 1 $, we have $ \,  \psi^0(h)  = a_0 \,\,r_-^2 \, \sin^2 \theta \, \cos^2 \theta \, (1 - m \, \sin^2 \theta) $.
    The upper bound  $ \, \varphi_0 \, $ of the     integral is associated to the relation $ \, {1\over2} = r_- \, \sin^2 \varphi_0 \, $ and
    $ \,  \varphi_0 = \arcsin \big( {{1}\over{\sqrt{2 \, r_-}}} \big) $. We have finally
    $ \, , I_0 =  \int_0^{\,1/2} \!\!  {{\dd h}\over{\sqrt{\psi^0 (h)}}}
    =  \int_0^{\,\varphi_0} \! {{1}\over{\sqrt{a_0 \, r_+}}} \,   {{2 \, \dd \theta}\over{\sqrt{1 - m \, \sin^2 \theta}}} \, $ and the result is established.    
\end{proof} 

\begin{remark} \label{psi-zero-et-psi-un}
  As emphazised previously, the computation of the second term
  
  \centerline {$
    I_1 \equiv  \int_{1/2}^{\,1} \!\!  {{\dd h}\over{\sqrt{\psi^1 (h)}}}
    =  \int_{0}^{\,1/2} \!\!  {{\dd h}\over{\sqrt{\psi^1 (1-h)}}}   $}

  \noindent 
  of the global integral $\, I_{\lambda} = I_0 + I_1 \, $  is obtained from the change of parameters $\,  \alpha \longleftrightarrow (-\beta) \, $
in Propositions \ref{P1-par-morceaux-alpha-0}, \ref{P1-par-morceaux-alpha-strictement-positif} and \ref{P1-par-morceaux-alpha-strictement-negatif}. 
The different cases are illustrated on   Figure~\ref{psi-un}.
\end{remark} 
This ends the proof of Theorem \ref{theorem-analytical}.
\end{proof}

\section{Semi-analytical approximation}
\label{sec:semi-analytical} 

In this section, we revisit the results of Dellacherie {\it et al.} \cite{SYNASC16}, where an approximation of $\psi_{\lambda}$ is constructed thanks to the function $\frac12[h(1-h)-\lambda V_{{\tilde a}, {\tilde b}}]$, where $V_{{\tilde a}, {\tilde b}}$ is an approximation of a solution of $-X''= \Sigma$ with boundary conditions $X(0)=X(1)=0$. Let $Q_0$ and $Q_1$ be the polynomials $h(h-1)$ and $h^2(h-1)$ respectively. The function $V_{{\tilde a}, {\tilde b}}$ is obtained thanks to the finite element approximation in the space of dimension 2
$ W=\, \{ V_{a, b}= aQ_1+bQ_0 \, \}$ and we seek $({\tilde a},{\tilde b})$ such that
$$\forall V\in W, \int_0^1 V'_{{\tilde a},{\tilde b}}(h)V'(h)dh= \int_0^1 \Sigma(h)V(h)dh.$$
The system on the coefficients $({\tilde a},{\tilde b})$ is
$$\left\{\begin{array}{l} (\int_0^1 (Q'_1)^2dh) {\tilde a}+(\int_0^1 Q'_1Q'_0 dh){\tilde b}= \int_0^1 Q_1(h)\Sigma(h)dh:=S_1\\
(\int_0^1 Q'_1Q'_0 dh){\tilde a}+(\int_0^1 (Q'_0)^2dh) {\tilde b}= \int_0^1 Q_0(h)\Sigma(h)dh:=S_0\end{array}\right.$$
which writes $K\left(\begin{array}{l}{\tilde a}\\ {\tilde b}\end{array}\right)= \left(\begin{array}{l}S_1\\ S_0\end{array}\right)$, where $K=\left(\begin{array}{ll}\frac13&\frac16\\ \frac16&\frac{2}{15}\end{array}\right)$.
We consider again this specific case to be able to compare with the results of section
\ref{sec:analytical}. This specific case is split in two parts: \\
i) $\Sigma$ is the interpolation polynomial:  $\, \Sigma \, : \Sigma(h)= \sigma_0(2h-1)(h-1)+4\sigma_{\frac12}h(1-h)+\sigma_1(h(2h-1))$,\\
one finds

\centerline {$
   \widetilde{V_{{\tilde a}, {\tilde b}}} = [{1\over5} \, \big( (3\, \sigma_0 + 4\, \sigma_{1/2} -2 \, \sigma_1 )\, (1-h)
+ (-2\,  \sigma_0 + 4\, \sigma_{1/2} + 3 \, \sigma_1 )\, h \big)]h(1-h) $.}

ii) $ \, \Sigma \, $ is piecewise affine:\\
   \centerline   {$
    \Sigma(h) =  \left \{ \begin {array}{l} \sigma_0 \, (1 - 2\, h) + 2 \, \sigma_{1/2} \, h \,\,\,\,\,\,\,\,\,\,\,\, {\rm if} \, h \leq {1\over2} \\
      2 \, \sigma_{1/2} \, (h-1) + \sigma_1 \, ( 2\, h - 1)  \,\,\,\,\,\,\, {\rm if} \, h \geq {1\over2} \,. 
    \end{array} \right.  $}
    
     We obtain after some elementary calculus
   
\centerline {$
   \widetilde{\Sigma} = {1\over16} \, \big( (11\, \sigma_0 + 10 \, \sigma_{1/2} - 5 \, \sigma_1 )\, (1-h)
+ (-5\,  \sigma_0 + 10\, \sigma_{1/2} + 11 \, \sigma_1 )\, h \big) $.}

\noindent
Then the process follows analogously to the first case with an affine function. In \cite{Du24} an implementation in {\it Python} is proposed.

\section{Numerical method}
\label{sec:numerical}

In this section, we describe a general numerical method (valid for any continuous function $\, \Sigma $)
which solves directly the coupled problem (\ref{modele-couple}) and finds an equation
for computing $ \, \lambda_* \, $ introduced in lemma \ref{lem:hstar}. This numerical method needs only to solve one equation with one single real unknown. 
\begin{theorem}
The equation 
\begin{equation} \label{equation-lambda-discret}\sum_{j=0}^{N-1} {{h_{j+1} - h_j}\over{ {1\over2} \Big(\sqrt{\psi_\lambda(h_j)} + \sqrt{\psi_\lambda(h_{j+1})} \, \Big)}}  = 1 \, .
\end{equation}
where $ \, h_j \equiv \sin^2 \big( {{\pi \, j}\over{2 \, N}} \big) \, $
returns a unique $\lambda^N$ which converges to the exact value of $\lambda$.
\label{5.1}
\end{theorem}
\begin{proof}
Recall that problem (\ref{modele-couple}) is  
\begin{equation*} \left \{ \begin {array}{l}
    - \varphi''(z) + \varphi(z) = \lambda \, \Sigma \big(h(z) \big) \,  \varphi(z) ,\,   h'=  \varphi(z)   \,, \,
    0 < z < 1 \,, \\
    h(0) = 0 , \, h(1) = 1 , \, \varphi(0) = 0 ,\, \varphi(1) = 0 ,\, \lambda \geq 0 ,\, \varphi > 0.  
  \end{array} \right.  \end{equation*}
We introduce a nonregular meshing $\, 0 = z_0 < z_1 < \, $... $ < z_{N-1} < z_N = 1 \, $ of the interval $\, [0,\, 1]\, $
and we set $ \, \Delta z_{j+1/2} \equiv z_{j+1} - z_j \, $ for $ \, j = 1 ,\, $ ... $ \,, \, N$.
We integrate the differential equation  $ \,  h'(z) =  \varphi(z) \, $        with the Crank-Nicolson scheme:
\begin{equation}  \label{maillage-fev2025} {{h_{j+1} - h_j}\over{\Delta z_{j+1/2}}} = {1\over2} \, \big( \varphi_j + \varphi_{j+1} \big) \,\, {\rm for} \,
  j = 0 ,\, {\rm ...} \,,\, N-1 \, .  \end{equation}
Observe that the relation (\ref{maillage-fev2025}) defines the mesh through the steps $ \, \Delta z_{j+1/2} $. 
Then after two integrations (see equalitites (\ref{aaa})), the first equation can be written as 
$ \, \varphi_j = \sqrt{\psi_\lambda(h_j)} \, $ and the  Crank-Nicolson scheme takes the form
\begin{equation}  {{h_{j+1} - h_j}\over{ {1\over2} \, \Big(\sqrt{\psi_\lambda(h_j)} + \sqrt{\psi_\lambda(h_{j+1})} \,\Big)}} = \Delta z_{j+1/2}(\lambda)  \,  . \end{equation}
Note that the choice of the Crank-Nicolson algorithm allows to recover a discretization of each
sub-problem (namely the idealized neutronic one and the simplified thermo-hydraulic one), the discretization of the equation on $\varphi$ is chosen as

$$2\frac{\frac{\varphi_{j+1}-\varphi_j}{\Delta z_{j+\frac12}}-\frac{\varphi_{j}-\varphi_{j-1}}{\Delta z_{j-\frac12}}}{\Delta z_{j+\frac12}+\Delta z_{j-\frac12}}=\lambda\frac{\psi_{\lambda}(h_{j+1})-\psi_{\lambda}(h_{j-1})}{h_{j+1}-h_{j-1}}\frac{h_{j+1}-h_{j-1}}{\Delta z_{j+\frac12}+\Delta z_{j-\frac12}},$$
for $j=1, ..., N-1$.
We impose the values $ \, h_j \equiv \sin^2 \big( {{\pi \, j}\over{2 \, N}} \big) \, $ in order to take into account the singularities
and two boundary conditions of the problem at $ \, z = 0 \, $ and $ \, z = 1 $.
The notation $ \,  \Delta z_{j+1/2}(\lambda) \, $ in the right hand side of the previous relation
is justified by the fact that if the numbers $ \, h_j \, $ are given, the left hand side is a simple function
of the scalar $ \, \lambda $. 
The number $ \, \lambda \, $ is {\it a priori} not known, but we have the natural relation
$ \, \sum_{j=0}^{N-1}  \Delta z_{j+1/2} (\lambda) = 1 \,$ which is (\ref{equation-lambda-discret}).

Recall that $\psi_{\lambda}$ and $V$ are respectively introduced in (\ref{psilambda}) and (\ref{Vfunction}).

The proof of the Theorem is then a consequence of
\begin{lemma}
  Assume $\Sigma\in L^{\infty}$ (such that $V\in C^1[0,1]$).
Define $$\lambda_{max}=\frac12\mbox{sup}_{[0,1]}\frac{h(h-1)}{V(h)}.$$ 

(i) For $\lambda>\lambda_{max}$, $\psi_{\lambda}(h_j)>0$ for all $j$.

  (ii) Equation (\ref{equation-lambda-discret}) has a unique solution $\lambda\in(\lambda_{max}, +\infty)$ when
  
\begin{equation}\label{condition}
\sum_{j=0}^{N-1}{{h_{j+1} - h_j}\over{ {1\over2} \big(\sqrt{\psi_{\lambda_{max}}(h_j)} + (\sqrt{\psi_{\lambda_{max}}(h_{j+1})} \,  \big)}}\in {\R}_+>1.\end{equation}
\end{lemma}
\begin{proof}
We check first that the function $h\rightarrow \frac{h(h-1)}{2V(h)}$ can be extended by continuity at $h=0$ and $h=1$ and is $C^0$ on $[0, 1]$, hence it has at least a point of maximum in $[0, 1]$.\\
  (i) One notes that $V<0$ on $[0,1]$ thanks to $\Sigma\geq \Sigma_*>0$. As $\psi_{\lambda}(h)= (-2V(h))(\lambda-\frac12\frac{h(1-h)}{-V(h)})$, and $\frac{h(h-1)}{2V(h)}\leq \lambda_{max}$, $\psi_{\lambda}(h)\geq (-2V(h))(\lambda-\lambda_{max})$ hence is never 0 for $\lambda>\lambda_{max}$.\\
(ii) We remark that, for all $\lambda\geq \lambda_{max}$, there is no $j\in\{0, .., N-1\}$ such that $\sqrt{\psi_{\lambda}(h_j)}$ and $\sqrt{\psi_{\lambda}(h_{j+1})}$ vanish simultaneously. Hence the quantity\\ \centerline{$
\sum_{j=0}^{N-1}{{h_{j+1} - h_j}\over{ {1\over2} \big(\sqrt{\psi_{\lambda_{max}}(h_j)} + (\sqrt{\psi_{\lambda_{max}}(h_{j+1})} \,  \big)}}$}\\ is never infinite.\\
The quantity $I^N(\lambda)=
\sum_{j=0}^{N-1}{{h_{j+1} - h_j}\over{ {1\over2} \big(\sqrt{\psi_{\lambda}(h_j)} + (\sqrt{\psi_{\lambda}(h_{j+1})} \,  \big)}}$ goes to $0$ when $\lambda\rightarrow +\infty$ for a fixed $N$, and is decreasing to 0 when $\lambda\rightarrow +\infty$. Hence a necessary and sufficient condition for having $\lambda^N>\lambda_{max}$ such that $I^N(\lambda^N)=1$ is $I^N(\lambda_{max})>1$. Item (ii) is proved.\\
\end{proof}
This ends the proof of Theorem \ref{5.1}.
\end{proof}
It could happen that $\sqrt{\psi_{\lambda_{max}}(h_j)} + \sqrt{\psi_{\lambda_{max}}(h_{j+1})}=0$, and we have to deal with these cases, which occur if and only if there exists $j$ such that $\psi_{\lambda_{max}}(h_j)=\psi_{\lambda_{max}}(h_{j+1})=0$. If $j$ is one of the points such that $\mbox{argmax}_{k}\frac{h_k(h_k-1)}{2V(h_k)}=j$ we have two cases:\\
 If a neighbor $j'$ of $j$ yields also a point of maximum, then there exist a solution to (\ref{equation-lambda-discret}) in $(\frac{h_j(h_j-1)}{2V(h_j)}, +\infty)$ because the function $\frac{\vert h_{j'}-h_j\vert}{\sqrt{\psi_{\lambda}(h_j)} + \sqrt{\psi_{\lambda}(h_{j+1})}}$ goes to $+\infty$ when $\lambda\rightarrow \mbox{max}_{k}\frac{h_k(h_k-1)}{2V(h_k)}:=\lambda^{N, j}$.\\
  If $j$ is unique, the necessary and sufficient condition for existence of a solution $\lambda$ is $I^N(\lambda^{j, N})>1$. This condition is less easy to check than $I^N(\lambda_*)>1$.

%
%
%
For a fixed discretization with $ \, N \, $ mesh points, a Newton algorithm
can be implemented without difficulty. With this procedure,  we recover on one hand  an approximated value $ \, \lambda_N \, $ of the 
unknown $ \, \lambda \, $ and on the other hand the entire approximate  solution of the problem $ \, h_j \approx h(z_j) \, $ and
$ \, \varphi_j \approx \sqrt{\psi_\lambda (h_j)} $. Observe that at convergence of the Newton algorithm, the abscissae $ \, z_j\, $
are function of the solution~$ \, \lambda_N \, $ and the converged space mesh is a result of the problem.
This coupled problem can be reduced to a single equation with only one real variable even after discretization!

\section{Numerical results.}
\label{sec:results}

{\it Toy model and its results:}\\
Consider the decreasing data $ \, \sigma_0=8 $, $ \, \sigma_{1/2} = 6 \, $ and $ \, \sigma_1 = 3 $.
They are presented in Figures \ref{double-psi-phi-hh2} and \ref{double-psi-phi-hh}, based on functions $\Sigma$ given in Figure \ref{sigma}.
We obtain the following exact values for the scalar parameter:
$ \, \lambda = 1.89036 \, $ in the decoupled case (case 0), 
$ \, \lambda = 1.99533 \, $ in the affine case  (case 1), 
$ \, \lambda = 1.86593 \, $ in the parabolic case  (case 2), 
$ \, \lambda = 1.89454 \, $ in the piecewise affine case (case 3) and 
$ \, \lambda = 1.85769 \, $ in the parabolic case approached by an affine functions (semi-analytical case 1) 
$ \, \lambda = 1.88614 \, $ in the piecewise affine case projected on affine functions (semi-analytical case 2). 
For each of these six cases, we have observed that our numerical approach gives converging results at second order accuracy in $O(\frac{1}{N^2})$
for the parameter $ \, \lambda $.\\
{\it Application to fictitious neutronical data by homogeneity:}

The problem depends only on $\lambda \Sigma(h)$, hence on $\lambda \sigma_0, \lambda \sigma_{1/2}, \lambda \sigma_1$, hence multiplying the quantities $\sigma_j, j=0, 1/2, 1$ by a constant $k_0$ and dividing $\lambda$ by the same constant does not change the problem we have to solve.

In neutronics problems, such as the equation ($\dagger$), a stationary solution is possible if and only if one finds a solution of $-(D\phi_0')'+\Sigma_a\phi_0=\nu \Sigma_f\phi_0$, that is $k=1$ in (\ref{neutronics}). If $k>1$ in neutronics, the PDE ($\dagger$) has a solution which is exponentially growing in time (divergence) and if $k<1$, the PDE ($\dagger$) has a solution which is exponentially decreasing in time (extinction). We observe  that, in our indealized, unphysical model problem, $\lambda$ is close to $2$. To be close to a reasonable situation, one uses the scaling of $\sigma_j$ by $c_0$ which induces the scaling of $\lambda$ by $\frac{1}{c_0}$ and we choose $c_0=1.86593$. If one wants to relate these calculations with the usual problems solved in the research centers studying nuclear reactors, $ \, \sigma_0 = 14.92744 $,
$ \, \sigma_{1/2} = 11.19558 \, $ and $ \, \sigma_1 = 5.59779 $. If we revert to the notation of (\ref{neutronics}), making this choice of values of $\sigma_i$ for $\nu \Sigma_f$ returns a value of $k$ close to 1, exactly equal to 1 in the case called 1 (affine case). \\
This yields, in this renormalized equivalent case the following exact values for $ \, k^* \, $:\\
$ \, 0.98708 \, $ in the decoupled case (case 0),
$ \, 0.93515  \, $ in the affine  case (case 1),
$ \, 1.00000  \, $ in the parabolic   case (case 2),
$ \, 0.98490  \, $ in the piecewise affine  case (case~3),
$ \, 1.00444  \, $ in the parabolic case approached by an affine functions (semi-analytical case 1) and 
$ \, 0.98928  \, $ in the piecewise affine case projected on affine functions (semi-analytical case 2).
This investigation shows that sometimes, the exigence of accuracy of the operational calculations
\cite{bensch}, \cite{duder} could be lightened. Indeed, the benchmarks generally used in the neutronicians community relies on precisions on $k^*$ (based on the calculation of $(k^*-1)/k^*$, and an acceptable precision between two benchmarks correspond to a difference in $(-15{\rm pcm}, 15{\rm pcm})$ (pcm: (per cent mille)), while the results presented here present, for neutron fluxes $\varphi$ extremely similar, variations much larger, and these variations are only related to different ways of approximating the cross section $\Sigma$ by different relevant approximations with the same values at $0, \frac12, 1$.\\
Observe nevertheless that we considered in this contribution an idealized model, we refer to \cite{SYNASC16} for a reference problem with actual data.\\
One can note that the macroscopic cross section that we consider here is rather steep, which could increase the different values of $(k^*-1)/k^*$, nevertheless for each couple of cases $p, q$, the difference $\vert\vert  \Sigma^p-\Sigma^q\vert\vert_{\infty}$ and $\vert\vert \varphi^p-\varphi^q\vert\vert_{H^1([0, 1])}$ are small. For example, the difference between $\Sigma$ for case 2 and 3 is $\frac18\vert \sigma_0+\sigma_1-2\sigma_{1/2}\vert$, in the toy model this provides $\frac18$, that is $2,2\%$ of the mean value. The observations on $\varphi$ are based on  Figure \ref{double-psi-phi-hh2} and not on norm study.

\begin{figure}
  \centerline{{\includegraphics[width = .50 \textwidth] {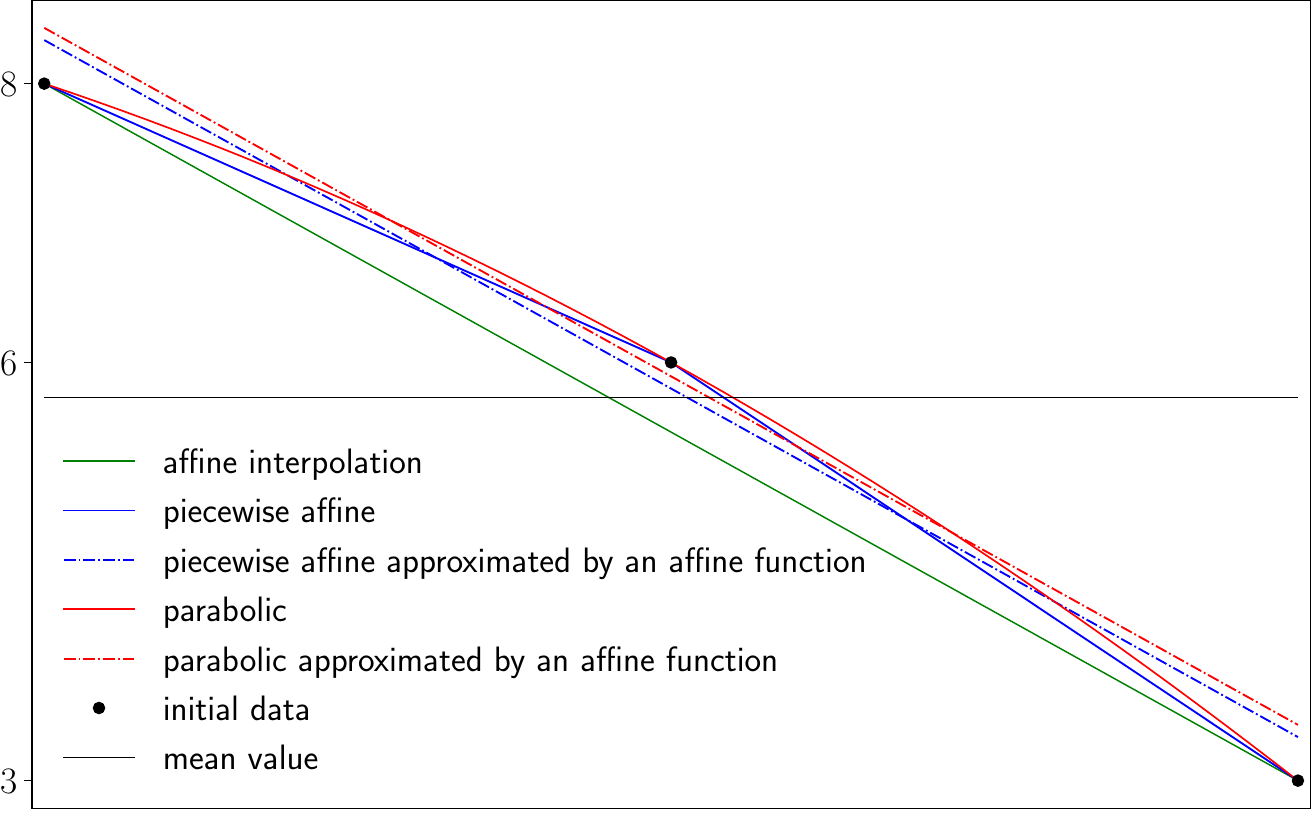}} {\includegraphics[width = .50 \textwidth] {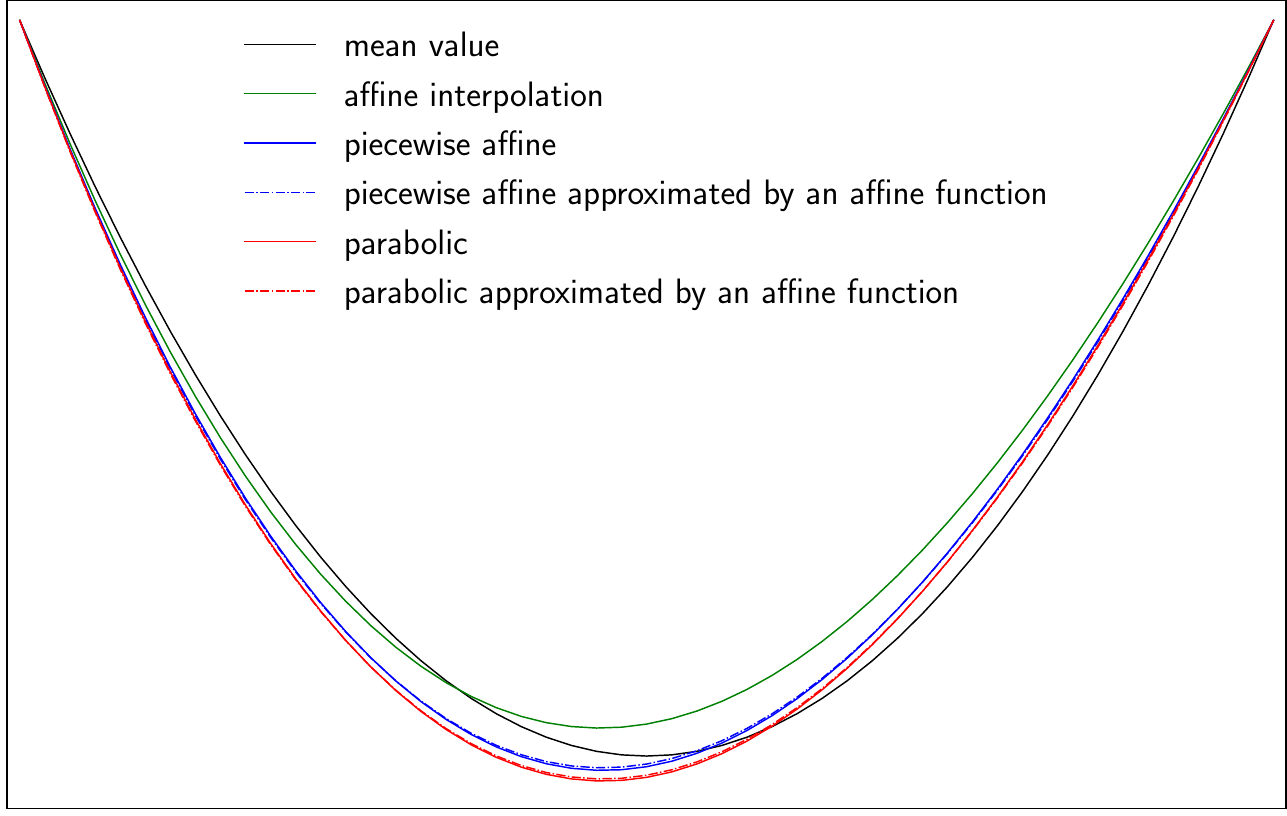}}}
    \caption{Function $\, \Sigma \, $ and $ \, V \, $ in all cases.}
    \label{sigma}
\end{figure}
\begin{figure}
  \centerline{{\includegraphics[width = .48 \textwidth] {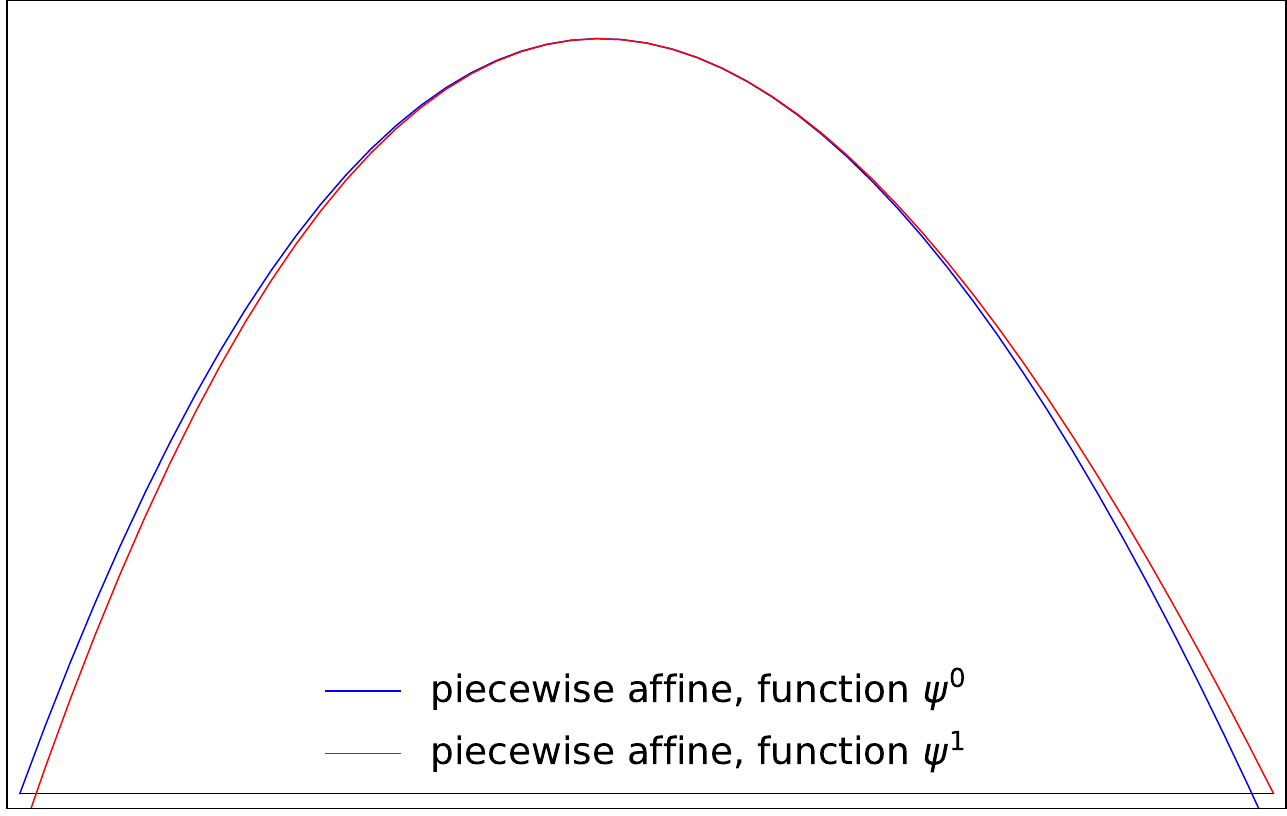}}
    {\includegraphics[width = .48 \textwidth] {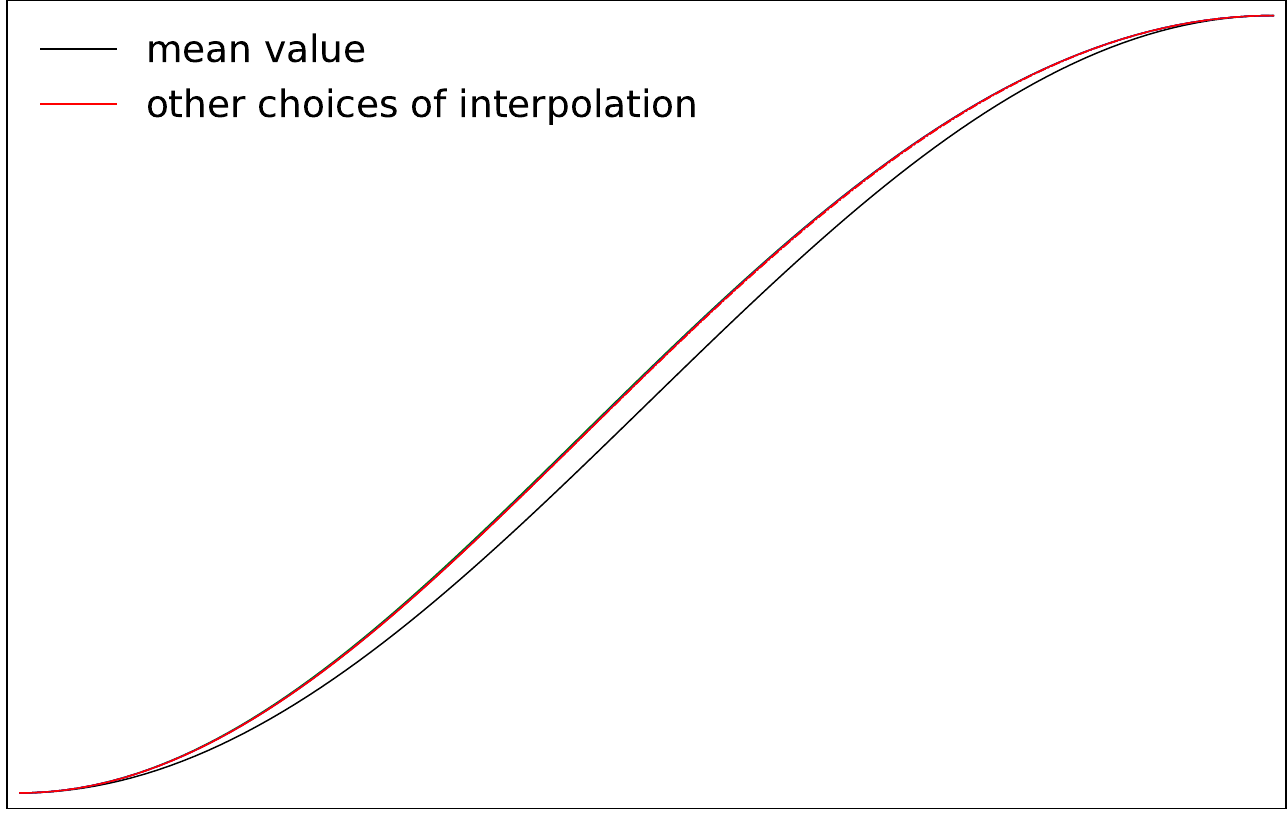}}}
  \caption{Functions $\, h \longmapsto \psi^0(h) \, $ and $ \,  h \longmapsto \psi^1(h)  \, $
    defined at relation (\ref{psi0-psi1-fev2025})  on the left. Function   $ \,  z \longmapsto  h(z)  \, $ on the right;
    a set if 5 curves are superimposed corresponding to affine interpolation, piecewise affine interpolation, piecewise affine approximated by an affine function,
    parabolic interpolation, and  parabolic approximated by an affine function.}
    \label{double-psi-phi-hh}
\end{figure}

\begin{figure}
  
{\includegraphics[width = .60 \textwidth] {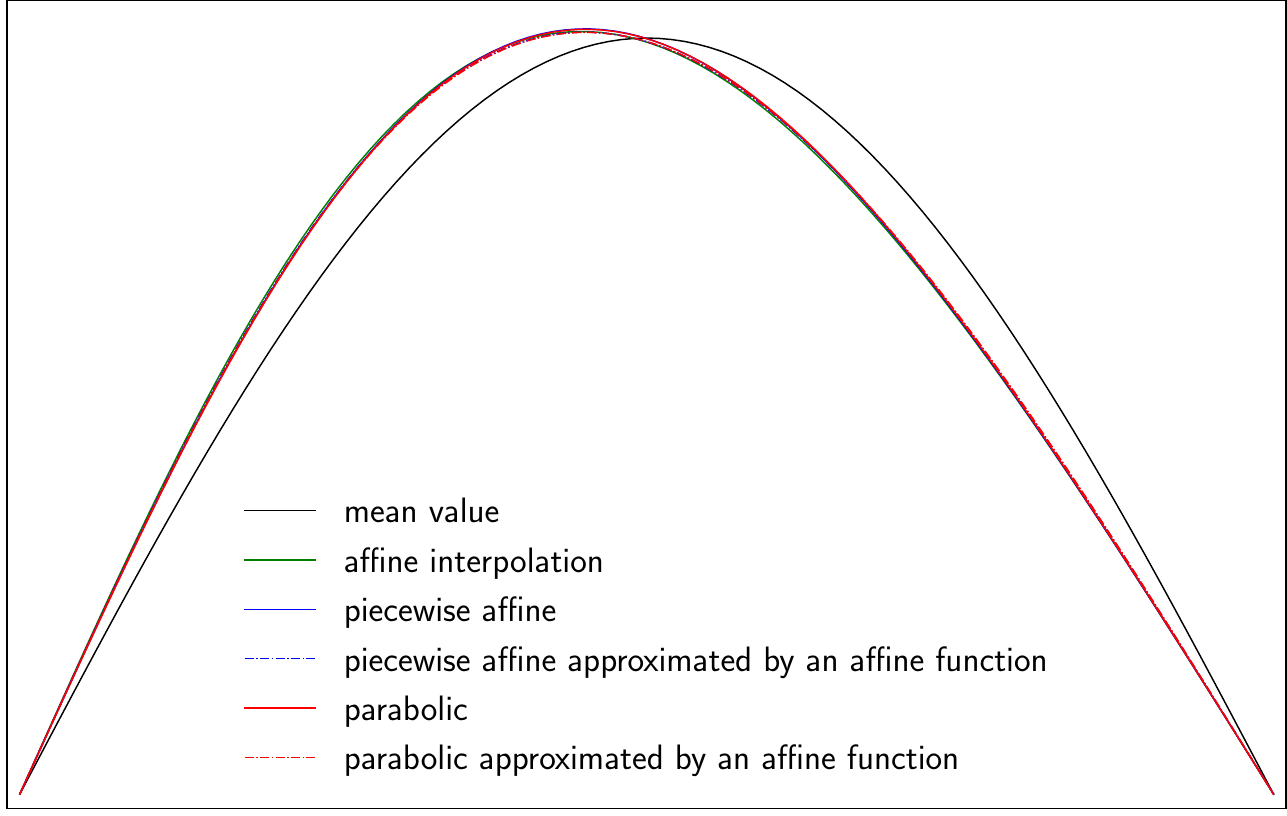}}
   
  \caption{Functions 
    $\,   z \longmapsto  \varphi (z) \, $ in all cases listed.}
    \label{double-psi-phi-hh2}
\end{figure}

\section{Conclusion}
\label{sec:conclu}

We considered in this paper a simplified idealized one-dimensional coupled model for neutronics and thermo-hydraulics.
A numerical method computing the unique solution of this coupled model $(k, \varphi, h)$, $ \, k \, $ being the multiplication factor, $\varphi$ being the neutron flux profile, and $h$ being the enthalpy of the fluid is based on the Crank-Nicolson scheme. Another point of view (namely evaluation incomplete elliptic integrals, which are well known special functions) allows to obtain analytical results, which show that the numerical method is extremely accurate. We observe important (meaning more than the usual range generally used as benchmarks) differences in the multiplication factor, even
if the neutron flux is really similar. \\
Future work concerns, for example, increasing the number of discretization points of cross sections (the extension of the numerical approach is straightforward, the extension of the analytical method could be possible (see Sections 17.1.1 to 17.1.5 of \cite{AS}), or use of the analytic formulae for sensitivity analysis.

\section{Appendix}
The possible cases for the elliptic integrals considered in this paper are:
\begin{lemma} \label{integrales-elliptiques}
(i)    Given two positive reals $ \, a \, $ and $ \, b \, $ such that $ \, 0 < a < b $, we have
\begin{equation}\label{cas-racines-plus-grandes-que-un-ou-negatives}
  \int_{-a}^{\,a} {{\dd T}\over{\sqrt{(a^2 - T^2) \, (b^2 - T^2)}}}  = {{2}\over{b}} \, K(m) \,,\,\, m = {{a^2}\over{b^2}} \, .  
\end{equation}
\begin{equation}\label{cas-racines-de-part-et-dautre}
  \int_{a}^{\,b} {{\dd T}\over{\sqrt{(T^2 - a^2) \, (b^2 - T^2)}}}  = {{1}\over{b}} \, K(m) \,,\,\, m = 1 - {{a^2}\over{b^2}} \, .  
\end{equation}
(ii)   Given a positive real  $ \, a \, $ and a non null real number $ \, b $, we have
\begin{equation}\label{cas-sigma-egal-un}
  \int_{-a}^{\,a} {{\dd T}\over{\sqrt{(a^2 - T^2) \, (b^2 + T^2)}}}  = {{2}\over{|b|}} \, K(m) \,,\,\, m = -{{a^2}\over{b^2}} \, .  
\end{equation}
\end{lemma}
\begin{proof}
  Cut the first integral into two equal parts, between $-a$ and $0$, and between $0$ and $a$, then introduce the change of variable $ \, T = a \, \sin \theta \, $ with $ \, 0 \leq \theta \leq  {{\pi}\over{2}}$. Then

  \centerline {$  \sqrt{a^2-T^2} = a \, \cos \theta $,
    $ \, \dd T = a \,  \cos \theta \, \dd \theta $, $ \,  \sqrt{b^2-T^2} = b \, \sqrt{1 - {{a^2}\over{b^2}} \, \sin^2 \theta}  $}

  \noindent
  and the first relation is established. 
  The same calculus conducts to

  \noindent $   \sqrt{b^2+T^2} = |b|  \, \sqrt{1 - \big(-{{a^2}\over{b^2}}\big) \, \sin^2 \theta} \, $ and the negative value for the parameter
$ \, m \, $ in the third relation is clear.
For the integral (\ref{cas-racines-de-part-et-dautre}), we consider the change of variables
$ \, T = b \, \sqrt{1-m\sin^2 \theta} \,$  with $\, T(0) = b \, $ and $ \, T({{\pi}\over{2}}) = a $. Then $ \, a = b \, \sqrt{1-m} \, $ and
$ \,  m = 1 - {{a^2}\over{b^2}} \in (0,\, 1) $. We have on one hand $ \, T \, \dd T = - b^2 \, m \, \sin \theta \, \cos \theta \, \dd \theta \, $
and on the other hand $ \, T^2 - a^2 = b^2 \, m \, \cos^2 \theta $. Then

  \noindent \qquad  $ 
\int_{a}^{\,b} \! {{\dd T}\over{\sqrt{(T^2 - a^2) \, (b^2 - T^2)}}}
= \int_{\pi/2}^{\, 0} \! \big({{1}\over{T}}\big) \, (- b^2 \, m \, \sin \theta \, \cos \theta) \,{{\dd \theta}\over{b \, \sqrt{m} \,\sin \theta \,\sqrt{m}  \cos \theta }} $

\noindent \qquad \qquad \qquad \qquad  $ 
= \int_0^{\,\pi/2} \! {{1}\over{b}} \, {{\dd \theta}\over{\sqrt{1-m\sin^2 \theta}}} $
$ = \smash{ {{1}\over{b}} \, K(m) }   $.
\noindent
The relation (\ref{cas-sigma-egal-un}) comes from $ \, T = a \, \sin \theta \, $  then $\theta\in (-\frac{\pi}{2}, \frac{\pi}{2})$ and  $ \, \dd T = a \,  \cos \theta \, \dd \theta $, $b^2+T^2= b^2(1-m\sin^2\theta)$ for $m=-\frac{a^2}{b^2}$, hence the result.
\end{proof}

\end{document}